\documentclass[12pt]{amsart}
\usepackage{amssymb,amsmath}
\usepackage{psfrag,graphicx}
\usepackage{pst-math}
\usepackage{pstricks}
\usepackage{pst-plot}
\usepackage{pst-slpe}
\usepackage{pst-all}
\def\cz{\varrho}
\def\H {{\rm H}}

\def\R {\mathbb{R}}
\def\N {\mathbb{N}}
\def\D {{\mathfrak D}}

\def\d{{\rm d}}

\def \l {\langle}
\def \r {\rangle}

\def \and {{\qquad\text{and}\qquad}}

\newtheorem{proposition}{Proposition}[section]
\newtheorem{theorem}[proposition]{Theorem}
\newtheorem{corollary}[proposition]{Corollary}
\newtheorem{lemma}[proposition]{Lemma}
\theoremstyle{definition}
\newtheorem{definition}[proposition]{Definition}
\newtheorem{remark}[proposition]{Remark}
\newtheorem{example}[proposition]{Example}
\numberwithin{equation}{section}
\def \au {\rm}
\def \ti {\it}
\def \jou {\rm}
\def \bk {\it}
\def \no#1#2#3 {{\bf #1} (#3), #2.}
\def \eds#1#2#3 {#1, #2, #3.}
\title[Elastically-coupled extensible double-beam systems]
{Steady states of elastically-coupled\\ extensible double-beam systems}

\author[F. Dell'Oro, C. Giorgi and V. Pata]
{Filippo Dell'Oro, Claudio Giorgi and Vittorino Pata}

\address{Politecnico di Milano - Dipartimento di Matematica
\newline\indent
Via Bonardi 9, 20133 Milano, Italy}
\email{filippo.delloro@polimi.it {\rm (F. Dell'Oro)}}
\email{vittorino.pata@polimi.it {\rm (V. Pata)}}

\address{Universit\`a degli Studi di Brescia - DICATAM
\newline\indent
Via Valotti 9, 25133 Brescia, Italy}
\email{claudio.giorgi@unibs.it {\rm (C.\ Giorgi)}}

\begin{document}

%%%%%%%%%%%%%%%%%%%%%%%%%%%%%%%%%%%%%%%%%%%%%%%%%
\begin{abstract}
Given $\beta\in\R$ and $\cz,k>0$, we analyze an abstract version of the nonlinear stationary model
in dimensionless form
\begin{align*}
\begin{cases}
u'''' - \Big(\beta+ \cz\int_0^1 |u'(s)|^2\,\d s\Big)u'' +k(u-v) = 0\\\noalign{\vskip0.8mm}
v'''' - \Big(\beta+ \cz\int_0^1 |v'(s)|^2\,\d s\Big)v'' -k(u-v) = 0
\end{cases}
\end{align*}
describing the equilibria of an elastically-coupled extensible double-beam
system subject to evenly compressive axial loads.
Necessary and sufficient conditions
in order to have nontrivial solutions are established, and their explicit closed-form expressions are found.
In particular, the solutions are shown to exhibit at most three nonvanishing Fourier modes.
In spite of the symmetry of the system, nonsymmetric solutions appear, as well as solutions
for which the elastic energy fails to be evenly distributed.
Such a feature turns out to be of some relevance in the analysis of the longterm dynamics,
for it may lead up to nonsymmetric
energy exchanges between the two beams, mimicking
the transition from vertical to torsional oscillations.
\end{abstract}
%%%%%%%%%%%%%%%%%%%%%%%%%%%%%%%%%%%%%%%%%%%%%%%%%

\maketitle

\begin{center}
\begin{minipage}{11cm}
\footnotesize
\tableofcontents
\end{minipage}
\end{center}
\newpage

%%%%%%%%%%%%%%%%%%%%%%%%%%%%%%%%%%%%%%%%%%%%%%%%%
\section{Introduction}

\subsection{Physical motivations}
For engineering purposes, the mathematical modeling process can be
viewed as the first step towards the analysis of both static and dynamic responses of actual mechanical structures.
Nevertheless, it relies on an idealization of the physical
world, and has limits of validity that must be specified.
For a given system, different models can be constructed,
the ``best" being the simplest one able to
capture all the essential features needed in the investigation.
Among others, models of
elastic sandwich-structured composites are experiencing an increasing interest in the literature, mainly
due to their wide use in sandwich panels and their applications in many branches of modern civil,
mechanical and aerospace engineering \cite{Zenkert}.
Sandwich structures are in general symmetric, and their variety
depends on the configuration of the core. Such devices are designed to have high bending stiffness
with overall low density \cite{Davies,L}.
In particular, sandwich beams, plates and shells are flexible elastic structures built up by
attaching two thin and stiff external layers (beams, plates or shells) to a homogeneously-distributed
lightweight and thick elastic core \cite{Pl}. Their interest, which is relevant in structural mechanics,
has been recently extended even to nanostructures (see e.g.\ \cite{CK} and references therein).

Models of elastic sandwich structures can be obtained by applying either
the Euler-Bernoulli theory for beams or the Kirchhoff-Love theory for thin plates.
In this context, several papers have been devoted to the mechanical properties of elastically-connected
double Euler-Bernoulli beams systems.
For instance, free and forced transverse
vibrations of simply supported double-beam
systems have been studied in \cite{KS,O,VOK}, while the articles \cite{ZLM,ZLWL} are concerned with
the effect of compressive axial load on free and forced oscillations.
Within the framework of nanostructures, axial instability and buckling of double-nanobeam systems have been analyzed in \cite{MA,WW}.

Once a model is established, the next step is to (possibly) solve the mathematical
equations, in order to discover the nature of the system response.
In fact, the main goal is to predict and control the actual dynamics.
To this end, the analysis of the steady states, and in particular of their closed-form expressions, becomes crucial.
This is even more urgent when dealing with nonlinear systems, where the longterm dynamics
is strongly influenced by the occurrence of a rich set of stationary solutions.

\subsection{The model}
In this paper, we aim to classify the stationary solutions, finding their explicit closed-form expressions, to
symmetric elastically-coupled extensible double-beam systems. For instance,
a sandwich structure composed of two elastic beams bonded
to an elastic core (Fig.\,\ref{Fig_1}a), or
the road bed of a girder bridge composed of an elastic rug connecting two
lateral elastic beams (Fig.\,\ref{Fig_1}b).
In both cases, the mechanical structure can be described by means of
two equal beams complying with the nonlinear model of Woinowsky-Krieger \cite{W}, which
takes into account extensibility, so that large deformations are allowed.
The beams are supposed to have the same natural length $\ell>0$, constant mass density,
and common thickness $0<h\ll\ell$.
At their ends, they are simply supported
and subject to evenly distributed axial loads.
A system of linear springs models the elastic filler connecting the beams:\
when the system lies in its natural configuration,
the beams are straight and parallel. The distance between the beams is equal to the free lengths of the springs.
Denoting by $\nu\in(-1,\tfrac12)$ the Poisson ratio of the beams,
the dynamics of the resulting undamped model is ruled by the following nonlinear
equations in dimensionless form (see the final Appendix for more details about the derivation of the model)
\begin{equation}
\label{sistema}
\begin{cases}
\displaystyle
\frac{\ell (1-\nu)}{h}\Big(\partial_{tt}
- \frac{h^2}{12\ell^2}\partial_{ttxx}\Big)u +
\delta\partial_{xxxx} u
-\big(\chi +\|\partial_x u \|^2\big)\partial_{xx} u +\kappa (u-v)=0,\\\noalign{\vskip1.7mm}
\displaystyle
\frac{\ell (1-\nu)}{h}\Big(\partial_{tt}
- \frac{h^2}{12\ell^2}\partial_{ttxx}\Big) v + \delta\partial_{xxxx} v
-\big(\chi +\|\partial_x v \|^2\big)\partial_{xx} v -\kappa(u-v)=0,
\end{cases}
\end{equation}
having set
$$\|f\|=\bigg(\int_0^1 |f(s)|^2\,\d s\bigg)^\frac12.$$
In the vertical plane ($x$-$z$), system~\eqref{sistema} describes the in-plane downward
rescaled deflections of the midline of the beams\footnote{The functions
$u,v$ are appropriate rescaling of the original vertical deflections
of the midline of the two beams
$$
U,V:[0,\ell]\times \R^+ \to \R,
$$
in comply with the dimensionless character of system~\eqref{sistema}.
See the Appendix for more details.}
$$u,v:[0,1]\times\R^+\to\R$$
with respect to their natural configuration (see Fig.\,\ref{Fig_1}a).
It may be also used to describe out-of-plane rescaled deflections
of the same double-beam structure, accounting for both vertical and torsional oscillations
(see Fig.\,\ref{Fig_1}b). In the latter situation, each beam is assumed to swing in a
vertical plane and the lateral movements are neglected.
The structural constants $\delta,\kappa>0$ are related to
the common flexural rigidity of the beams and the common stiffness of the inner elastic springs,
respectively, whereas the parameter $\chi\in\R$ summarizes the effect of the axial force acting at the right ends of the
beams:\ positive when the beams are stretched, negative when compressed.

%%%%%%%%%%%%%%%%%%%%%%%%%%%%%%%%%%%%%%%%%%%%%%
\begin{figure}[ht]
\includegraphics[width=7.3cm]{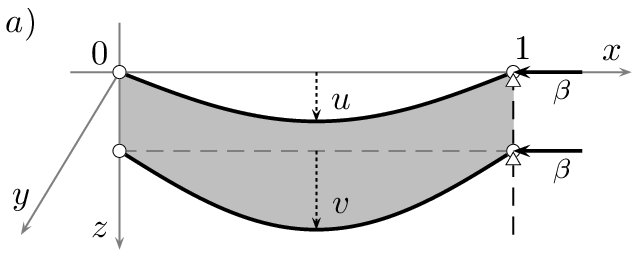}
\hspace{3.5mm}
\includegraphics[width=7.3cm]{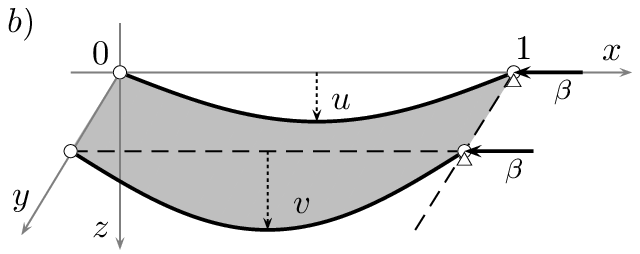}
\caption{In-plane (a) and out-of-plane (b) deflections of a double-beam system
under compressive axial loads $\beta={\chi}/{\delta}.$
\label{Fig_1}}
\end{figure}
%%%%%%%%%%%%%%%%%%%%%%%%%%%%%%%%%%%%%%%%%%%%%%

In this work, we are interested in the stationary solutions to
the evolutionary problem~\eqref{sistema}, subject to the hinged boundary conditions.
Namely, setting
$$\beta=\frac{\chi}{\delta}\in\R,\qquad \cz = \frac{1}{\delta}>0,\qquad k=\frac{\kappa}{\delta}>0,$$
we consider the dimensionless system of ODEs
\begin{equation}
\label{sistema_staz}
\begin{cases}
u'''' - \big(\beta+\cz\|u'\|^2\big)u''+k(u-v) = 0,\\
\noalign{\vskip1mm}
v''''- \big(\beta+\cz\|v'\|^2\big)v''-k(u-v) = 0,
\end{cases}
\end{equation}
supplemented with the boundary conditions
\begin{equation}
\label{BC_sistema}
\begin{cases}
u(0)=u(1)=u''(0)=u''(1)=0,\\\noalign{\vskip0.5mm}
v(0)=v(1)=v''(0)=v''(1)=0.
\end{cases}
\end{equation}
It is apparent that problem
\eqref{sistema_staz}-\eqref{BC_sistema} always admits the trivial solution $u=v=0$, while
the occurrence and the complexity of nontrivial solutions strongly depend on the values of
%%%%%%%%% modificato %%%%%%%%%
structural dimensionless parameters $\beta,\cz,k$, all of which are allowed to be large (see the final comment in the Appendix).
%%%%%%%%%%%%%%%%%%%%%%%%%

\subsection{Earlier results on single-beam equations}
When system  \eqref{sistema_staz} is uncoupled (i.e.\ in the limit situation when $k=0$), the analysis reduces to the one
of the single Woinowsky-Krieger beam
$$
u'''' - \big(\beta+ \cz\|u'\|^2\big)u''=0.
$$
In this case, it is well-known that an increasing compressive axial load leads to a series of fork bifurcations.
The critical values of $\beta$ at which bifurcations occur depend on the eigenvalues of the differential operator
(see e.g.\ \cite{B1,CZGP}). After exceeding these values, the axial compression is sustained in one of two states of
equilibrium:\ a purely compressed state with no lateral deviation (the trivial solution) or two symmetric laterally-deformed
configurations (buckled solutions). This is why the phenomenon is usually referred to as {\it buckling}.
Another interesting model, formally obtained by neglecting
the second equation of system \eqref{sistema_staz} and by taking $v\equiv0$ in the first one, reads
$$
u'''' - \big(\beta+ \cz\|u'\|^2\big)u''+ku=0,
$$
namely, a single Woinowsky-Krieger beam which relies on an elastic foundation.
In this case, bifurcations of the trivial solution split into two series,
whose critical values depend
also on the ratio $k$ between the parameters $\kappa$ and $\delta$ connected with the stiffness of
the foundation and the flexural rigidity of the beam \cite{BoV}.

\subsection{The goal of the present work}
Clearly, when the double-beam system \eqref{sistema_staz} is considered,
the picture becomes much more difficult.
To the best of our knowledge, in spite of the quite large number of papers about statics and dynamics of single
Woinowsky-Krieger beams (e.g.\ \cite{B1,BoV,CZ,CZGP,Dick,Dick2,EM,GPV,HM,RM}),
no analytic results concerning models with a coupling between two (or more) nonlinear beams
of this type are available in the literature.
This may be due to the fact that classifying and finding closed-form expressions for the solutions to
equations of this kind is in general a very difficult, if not impossible, task. Indeed, it is usually
unavoidable to replace distributed characteristics with discrete
ones, so producing approximate solutions by resorting to some discretization procedures. Unfortunately, this
strategy can be hardly applied when multiple stable states occur (see e.g.\ \cite{L} and references therein).

Here, our aim is to fill this gap.
To this end, we first recast \eqref{sistema_staz}-\eqref{BC_sistema} into an abstract nonlinear system
involving an arbitrary strictly positive selfadjoint linear operator $A$ with compact inverse.
Then, we classify all the nontrivial solutions, finding also their explicit expressions.
In particular, every solution is shown to exhibit at most three nonvanishing Fourier modes.
According to our classification, the set of stationary solutions to
nonlinear double-beam systems is very rich.
The nonlinear terms accounting for extensibility substantially
influence the instability (or buckling):\ the effects are higher with
increasing values of (minus) the axial-load parameter $\beta$, and give rise to both in-phase (synchronous)
buckling modes and out-of-phase (asynchronous) buckling modes. This feature becomes quite important
in the study of the longterm behavior, as it may lead up to nonsymmetric energy exchanges
between the two beams under small perturbations. In the
asymptotic dynamics of a double-beam structure like the road bed of a girder bridge (Fig.\,\ref{Fig_1}b),
a nonsymmetric energy exchange of this kind is apt to mimic
the transition from vertical to torsional oscillations, such as those occurred
in the collapse of the Tacoma Narrows suspension bridge (see e.g.\ \cite{McK} and references therein).
Another remarkable fact is that the model \eqref{sistema_staz} has been derived under the assumption
that the ratio $h/\ell$ between the thickness and the natural length of the beam is very small;
the critical values at which bifurcations occur are consistent with such an assumption, namely,
they are of order $h/\ell$ as well.
We also stress that system~\eqref{sistema_staz} is dimensionless, and no physical parameters have been artificially
set equal to one.
Finally it is worth noting that, as a consequence of the abstract formulation,
all the results are valid also for multidimensional structures. In particular, they are applicable
to flexible double-plate sandwich structures
with hinged boundaries, provided that the plates are modeled according to the Berger's approach \cite{ADR,KAM}.

\subsection{Plan of the paper}
In the next \S2 we introduce the aforementioned operator $A$, and we rewrite \eqref{sistema_staz}-\eqref{BC_sistema}
in an abstract form. In \S3 we prove that every solution can be expressed as a linear combination
of at most three distinct eigenvectors of $A$.
The subsequent \S4 deals with the analysis of unimodal solutions (i.e.\ solutions with only one eigenvector involved).
In particular, we show that not only a double series of
fork bifurcations of the trivial solution occur, but also buckled solutions
may suffer from a further bifurcation when $-\beta$ exceeds some greater critical
value. In \S5 we study the so-called
equidistributed energy solutions (i.e.\ solutions with evenly distributed
elastic energy), and we prove that bimodal and trimodal
steady states pop up.
In \S6 we classify the general (not necessarily equidistributed) bimodal solutions, while
in \S7 we show that every trimodal solution is necessarily an equidistributed energy solution,
The final \S8 is devoted to a comparison with some single-beam equations previously studied in the literature.
The derivation of the evolutionary physical model \eqref{sistema} is carried out in full detail in the concluding Appendix.
%%%%%%%%%%%%%%%%%%%%%%%%%%%%%%%%%%%%%%%%%%%%%%%%%

%%%%%%%%%%%%%%%%%%%%%%%%%%%%%%%%%%%%%%%%%%%%%%%%%
\section{The Abstract Model}

\noindent
Let $(\H,\l \cdot,\cdot \r, \|\cdot\| ) $ be a separable real Hilbert space, and let
$$A: \D(A)\Subset \H \to \H$$ be a strictly positive selfadjoint linear operator, where the
(dense) embedding $\D(A)\Subset \H$ is compact. In particular, the inverse $A^{-1}$ of $A$
turns out to be a compact operator on $\H$.
Accordingly, for $r\geq0$, we introduce the compactly nested family of Hilbert spaces
(the index $r$ will be omitted whenever zero)
$$
\H^r=\D(A^{\frac{r}2}),\qquad\l u,v\r_r=\l A^{\frac{r}2}u,A^{\frac{r}2}v\r,\qquad
\|u\|_r=\|A^{\frac{r}2}u\|.
$$
Then, given $\beta\in \R$ and $\cz,k>0$, we consider the abstract nonlinear stationary problem
in the unknown variables $(u,v)\in\H^2\times\H^2$
\begin{equation}
\label{MAIN}
\begin{cases}
A^2 u + C_u A u + k(u-v)=0,\\\noalign{\vskip0.7mm}
A^2 v + C_v Av -k(u-v)=0,
\end{cases}
\end{equation}
where
\begin{equation}
\label{Bautc}
C_u=\beta + \cz\|u\|_1^2 \and C_v=\beta + \cz\|v\|_1^2.
\end{equation}

\begin{definition}
A couple $(u,v)\in \H^2\times \H^2$ is called a {\it weak solution} to \eqref{MAIN} if
\begin{equation}
\label{weak}
\begin{cases}
\l u , \phi \r_2 + C_u \l u , \phi \r_1 + k\l (u-v) , \phi \r=0,\\\noalign{\vskip0.7mm}
\l v , \psi \r_2 + C_v \l v , \psi \r_1  -k\l (u-v) , \psi \r=0,
\end{cases}
\end{equation}
for every test $(\phi,\psi)\in \H^2 \times \H^2$.
\end{definition}

It is apparent that the trivial solution $u=v=0$ always exists.

\begin{example}
\label{dlo}
The concrete physical system \eqref{sistema_staz} is recovered by setting $\H=L^2(0,1)$ and
$A=L$, where
$$
L = -\frac{\d^2}{\d x^2} \qquad \text{with}\qquad \D(L)=H^2(0,1) \cap H_0^1(0,1).
$$
Here $L^2(0,1)$, as well as $H_0^1(0,1)$ and $H^2(0,1)$, denote the usual Lebesgue and Sobolev spaces
on the unit interval $(0,1)$.
In particular
$$
\H^2 = H^2(0,1) \cap H_0^1(0,1)\,\,\,\Subset\,\,\, \H^1 = H_0^1(0,1) \,\,\,\Subset\,\,\, \H=L^2(0,1).
$$
\end{example}

\smallskip
\noindent
{\bf Notation.}
For any $n\in\N = \{1,2,3,\ldots\}$ we denote by
$$
0<\lambda_n\to\infty
$$
the increasing sequence of eigenvalues of $A$, and by $e_n\in \H$ the corresponding normalized eigenvectors,
which form a complete orthonormal basis of $\H$. In this work, all the eigenvalues $\lambda_n$
are assumed to be {\it simple}, which is certainly true for the concrete realization
$A=L$ arising in the considered
physical models. Indeed, in such a case, the eigenvalues are equal to
$$
\lambda_n = n^2\pi^2
$$
with corresponding eigenvectors
$$
e_n(x)= \sqrt{2} \sin{(n\pi x)}.
$$
%%%%%%%%%%%%%%%%%%%%%%%%%%%%%%%%%%%%%%%%%%%%%%%%%

%%%%%%%%%%%%%%%%%%%%%%%%%%%%%%%%%%%%%%%%%%%%%%%%%
\section{General Structure of the Solutions}

\noindent
In this section we provide two general results on the solutions to system~\eqref{MAIN}.
To this end, we introduce the set of {\it effective modes}
$$
\mathbb{E} = \{ n : \lambda_n<-\beta \}.
$$
Clearly,
\begin{equation}
\label{basic}
\mathbb{E} \not=\emptyset \quad \Leftrightarrow\quad \beta<-\lambda_1.
\end{equation}
Therefore, if $\mathbb{E}\neq\emptyset$,
$$
\mathbb{E} = \{1,2,\ldots,n_\star \},
$$
where\footnote{Here and in what follows
$|\mathbb{S}|$ denotes the cardinality of a set $\mathbb{S}\subset\N$.}
$$
n_\star = \max \{n:\lambda_n<-\beta \}= |\mathbb{E}|.
$$

\begin{example}
When $A=L$ (the Laplace-Dirichlet operator introduced in the previous section), we have
$$
\mathbb{E} = \big\{ n : n^2 \pi^2<-\beta\big\}.
$$
Accordingly, in the nontrivial case $\beta<0$,
$$
|\mathbb{E}| = \left\lceil \sqrt{-\frac{ \beta}{\pi^2}} \,\,\right\rceil -1,
$$
the symbol $\lceil a \rceil$ standing for the smallest integer greater than or equal to $a$.
\end{example}

We begin to prove that the picture is trivial whenever the set $\mathbb{E}$ is empty.

\begin{proposition}
\label{trivial}
If $\mathbb{E}=\emptyset$ system \eqref{MAIN} admits only the trivial solution.
\end{proposition}

\begin{proof}
Let $(u,v)$ be a weak solution to \eqref{MAIN}. Choosing $(\phi,\psi)=(u,v)$ in the weak formulation \eqref{weak},
and adding the resulting expressions, we obtain the identity
$$
\|u\|_2^2 + \|v\|_2^2 + (\beta + \cz\|u\|_1^2)\|u\|_1^2 + (\beta + \cz\|v\|_1^2)\|v\|_1^2 + k\|u-v\|^2=0.
$$
Then, exploiting the Poincar\'e inequality
$$
\lambda_1\|w\|_1^2 \leq \|w\|_2^2,\quad \forall w \in \H^2,
$$
we infer that
$$
(\lambda_1+\beta)(\|u\|_1^2 + \|v\|_1^2) + \cz\|u\|_1^4 + \cz\|v\|_1^4+ k\|u-v\|^2\leq0,
$$
and, since $\lambda_1+\beta\geq0$, we conclude that $u=v=0$.
\end{proof}

Accordingly, from now on
we will assume (often without explicit mention) that \eqref{basic} be satisfied.
As it will be clear from the subsequent analysis, this condition turns out to be sufficient
as well in order to have nontrivial solutions. Hence, {\it a posteriori}, we can reformulate
Proposition \ref{trivial} by saying that system \eqref{MAIN} admits nontrivial solutions if and only
if the set $\mathbb{E}$ is nonempty.

\smallskip
The next result shows that every weak solution can be written as linear combination of at most
three distinct eigenvectors of $A$.

\begin{lemma}
\label{gene}
Let $(u,v)$ be a weak solution of system \eqref{MAIN}. Then
$$
u = \sum_n \alpha_{n} e_{n}
\and v = \sum_n \gamma_{n} e_{n}
$$
for some $\alpha_n,\gamma_n\in\R$, where $\alpha_n\not=0$ for at most three distinct values of $n\in\N$.
Moreover,
$$\alpha_n=0\quad \Leftrightarrow\quad \gamma_n=0.$$
\end{lemma}

\begin{proof}
Let $(u,v)$ be a weak solution to \eqref{MAIN}. Then, writing
$$
u = \sum_{n} \alpha_n e_n \and v = \sum_{n} \gamma_n e_n
$$
for some $\alpha_n,\gamma_n\in\R$,
and choosing $\phi=\psi=e_n$ in the weak formulation \eqref{weak}, we obtain for every $n\in\N$ the system
\begin{equation}
\label{ognin}
\begin{cases}
\lambda_n^2 \alpha_n + C_u\lambda_n\alpha_n + k(\alpha_n-\gamma_n)=0,\\
\lambda_n^2 \gamma_n + C_v\lambda_n\gamma_n - k(\alpha_n-\gamma_n)=0.
\end{cases}
\end{equation}
It is apparent that
$$\alpha_n=0 \quad\Leftrightarrow\quad \gamma_n=0.$$
Substituting the first equation into the second one, we get
$$
\gamma_n(\lambda_n^2 + C_v\lambda_n + k)(\lambda_n^2 + C_u\lambda_n + k) = k^2\gamma_n.
$$
Hence, if $\gamma_n\not=0$ (and so $\alpha_n\not=0$), we end up with
$$
\lambda_n^3 + (C_u+C_v)\lambda_n^2  + (C_uC_v + 2k)\lambda_n + k(C_u+C_v) = 0.
$$
Since the equation above admits at most three distinct solutions
$\lambda_{n_i}$ we are done.
\end{proof}

Summarizing, every weak solution $(u,v)$ can be written as
\begin{equation}
\label{espressione}
u=\sum_{i=1}^3 \alpha_{n_i} e_{n_i} \and v=\sum_{i=1}^3 \gamma_{n_i} e_{n_i},
\end{equation}
for three distinct $n_i\in\N$ and some coefficients $\alpha_n{_i},\gamma_{n_i}\in\R$.
In particular, from \eqref{Bautc}, we deduce the explicit expressions
\begin{equation}
\label{Bautc1}
C_u = \beta + \cz\sum_{i=1}^3 \lambda_{n_i}\alpha_{n_i}^2\and
C_v = \beta + \cz\sum_{i=1}^3 \lambda_{n_i}\gamma_{n_i}^2.
\end{equation}
In addition, when
$$\alpha_{n_i}\not=0\quad \Leftrightarrow\quad \gamma_{n_i}\not=0,$$
the corresponding eigenvalue $\lambda_{n_i}$
is a root of the cubic polynomial
$$
P(\lambda)= \lambda^3 + (C_u+C_v)\lambda^2  + (C_uC_v + 2k)\lambda + k(C_u+C_v).
$$
Notably, when the equality $C_u=C_v$ holds, the polynomial $P(\lambda)$ can be written in the simpler form
$$
P(\lambda) = (\lambda+C_u)(\lambda^2 +C_u\lambda + 2k).
$$

\begin{remark}
Adding the two equations of system \eqref{ognin}, we infer that
\begin{equation}
\label{rel}
\lambda_n = - \frac{C_u\alpha_n+C_v\gamma_n}{\alpha_n+\gamma_n}
\end{equation}
whenever $\alpha_n+\gamma_n\not=0$. This relation will be crucial for our purposes.
\end{remark}

As an immediate consequence of Lemma \ref{gene}, we also have

\begin{corollary}
Every weak solution $(u,v)$ is actually a strong solution. Namely,
$(u,v)\in \H^4 \times \H^4$ and \eqref{MAIN} holds.
Even more so, $(u,v)\in \H^r\times \H^r$ for every $r$.
\end{corollary}

\begin{remark}
In the concrete situation when $A=L$, every
weak solution $(u,v)$ is regular, that is, $(u,v)\in\mathcal{C}^\infty([0,1])\times\mathcal{C}^\infty([0,1])$.
\end{remark}

Finally, in the light of Lemma \ref{gene}, we give the following definition.

\begin{definition}
We call a solution $(u,v)$ {\it unimodal}, {\it bimodal} or {\it trimodal} if it involves one, two
or three distinct eigenvectors, that is,
if $\alpha_{n}\not=0$ (and so $\gamma_{n}\not=0$) for one, two or three indexes $n$, respectively.
\end{definition}
%%%%%%%%%%%%%%%%%%%%%%%%%%%%%%%%%%%%%%%%%%%%%%%%%

%%%%%%%%%%%%%%%%%%%%%%%%%%%%%%%%%%%%%%%%%%%%%%%%%
\section{Unimodal Solutions}
\label{sezioneUnimodal}

\noindent
We now focus on unimodal solutions. More precisely, we look for solutions $(u,v)$ of the form
\begin{equation}
\label{form}
\begin{cases}
u = \alpha_n e_n,\\
v = \gamma_n e_n,
\end{cases}
\end{equation}
for a fixed $n\in\N$ and some coefficients $\alpha_n,\gamma_n\not=0$.
In order to classify such solutions, we introduce the positive
sequences\footnote{Observe that $\lambda_n<\mu_n<\nu_n$.}
$$
\mu_n = \frac{2k}{\lambda_n} + \lambda_n\and
\nu_n = \frac{3k}{\lambda_n} + \lambda_n,
$$
along with the (disjoint) subsets of $\mathbb{E}$
\begin{align*}
&\mathbb{E}_1=\{n: \lambda_n < -\beta \leq \mu_n \},\\
&\mathbb{E}_2=\{n: \mu_n < -\beta \leq \nu_n \},\\
&\mathbb{E}_3=\{n: \nu_n < -\beta \}.
\end{align*}
Clearly,
$$
\mathbb{E}_1 \cup \mathbb{E}_2 \cup \mathbb{E}_3 = \mathbb{E}.
$$
Then, we consider the real numbers (whenever defined)
\begin{equation}
\label{ALPHA}
\begin{cases}
\displaystyle\alpha_{n,1}^{\pm}=\pm \sqrt{\frac{-\beta-\lambda_n}{\cz{\lambda_n}}},
\\\noalign{\vskip1mm}
\displaystyle\alpha_{n,2}^{\pm}=\pm \sqrt{\frac{-\beta-\mu_n}{\cz{\lambda_n}}},
\\\noalign{\vskip1mm}
\displaystyle\alpha_{n,3}^{\pm}=\pm \sqrt{\frac{ -\beta +\mu_n - \nu_n - \lambda_n  + \sqrt{(\beta+\lambda_n + \mu_n - \nu_n)
(\beta+\nu_n)}}{2\cz\lambda_n}},\\\noalign{\vskip1mm}
\displaystyle\alpha_{n,4}^{\pm}=\pm \sqrt{\frac{ -\beta +\mu_n - \nu_n - \lambda_n - \sqrt{(\beta+\lambda_n + \mu_n - \nu_n)
(\beta+\nu_n)}}{2\cz\lambda_n}},
\end{cases}
\end{equation}
hereafter called {\it unimodal amplitudes},
or {\sc u}-{\it amplitudes} for brevity.
By elementary calculations, one can easily verify that
\begin{align*}
&\alpha_{n,1}^\pm \in \R \quad \Leftrightarrow\quad \lambda_n \leq -\beta,\\
&\alpha_{n,2}^\pm \in \R \quad \Leftrightarrow\quad \mu_n \leq -\beta,\\
&\alpha_{n,3}^\pm \in \R \quad \Leftrightarrow\quad \nu_n \leq -\beta,\\
&\alpha_{n,4}^\pm \in\R \quad \Leftrightarrow\quad \nu_n \leq -\beta.
\end{align*}

\begin{lemma}
\label{reale}
For every fixed $n\in\N$, let us consider the set
$$
\Gamma_n = \{\alpha_{n,i}^\pm : i=1,2,3,4\}.
$$
Then, $\Gamma_n$ contains exactly
\begin{itemize}
\item 2 distinct nontrivial {\sc u}-amplitudes $\{\alpha_{n,1}^\pm\}$ if $n\in\mathbb{E}_1$;

\smallskip
\item 4 distinct nontrivial {\sc u}-amplitudes $\{\alpha_{n,1}^\pm,\alpha_{n,2}^\pm\}$
if $n\in\mathbb{E}_2$;

\smallskip
\item 8 distinct nontrivial {\sc u}-amplitudes $\{\alpha_{n,1}^\pm,\alpha_{n,2}^\pm,
\alpha_{n,3}^\pm,\alpha_{n,4}^\pm\}$ if $n\in\mathbb{E}_3$.
\end{itemize}
If $n\notin\mathbb{E}$, the set $\Gamma_n$ is either empty or
it contains exactly the (trivial) {\sc u}-amplitudes $\alpha_{n,1}^+=\alpha_{n,1}^-=0$.
\end{lemma}

\begin{proof}
We analyze separately all the possible cases.
\begin{itemize}
\item
If $n\in\mathbb{E}_1$,  there are only two distinct
nontrivial {\sc u}-amplitudes, that is, $\alpha_{n,1}^\pm$.
Indeed, when $\mu_n=-\beta$,
$$
\alpha_{n,2}^\pm = 0.
$$

\smallskip
\item If $n\in\mathbb{E}_2$,
there are only four distinct nontrivial {\sc u}-amplitudes, that is,
$\alpha_{n,1}^\pm$ and $\alpha_{n,2}^\pm$. Indeed, when $\nu_n=-\beta$,
$$
\alpha_{n,3}^+ = \alpha_{n,4}^+ = \alpha_{n,2}^+ \and \alpha_{n,3}^- = \alpha_{n,4}^- = \alpha_{n,2}^-.
$$

\smallskip
\item If $n\in\mathbb{E}_3$, all the eight {\sc u}-amplitudes $\alpha_{n,i}^\pm$ are
distinct and nontrivial.
\end{itemize}
If $n\notin\mathbb{E}$, all the {\sc u}-amplitudes $\alpha_{n,i}^\pm$, whenever defined, are trivial.
In particular, the only two allowed amplitudes are $\alpha_{n,1}^+=\alpha_{n,1}^-=0$.
\end{proof}

\newpage

%%%%%%%%%%%%%%%%%%%%%%%%%%%%%%%%%%%%%%%%%%%%%%
\begin{figure}[ht]
\includegraphics[width=10cm]{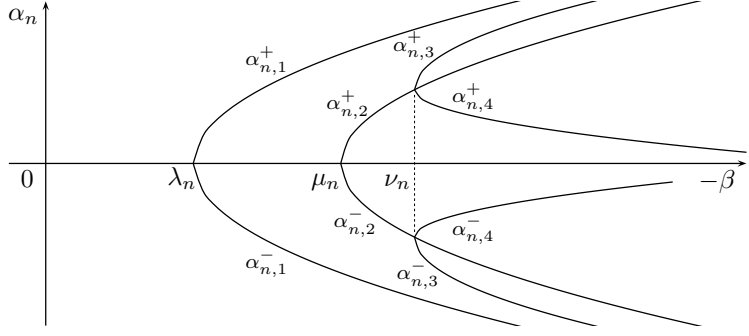}
\caption{The {\sc u}-amplitudes $\alpha_{n,i}^\pm$ for a fixed $n\in\N$.
\label{secondafigura}}
\end{figure}
%%%%%%%%%%%%%%%%%%%%%%%%%%%%%%%%%%%%%%%%%%%%%%

We are now in a position to state our main result on unimodal solutions.

\begin{theorem}
\label{unimodal}
System~\eqref{MAIN} admits nontrivial unimodal solutions if and only
if the set $\mathbb{E}$ is nonempty. More precisely, for every $n\in\N$, one of the following
disjoint situations occurs.

\begin{itemize}
\item If $n\in\mathbb{E}_1$, we have exactly 2 nontrivial unimodal solutions of the form
$$
(u,v) =
\begin{cases}
(\alpha_{n,1}^+\, e_n,\,\alpha_{n,1}^+\, e_n)\\
(\alpha_{n,1}^-\, e_n,\, \alpha_{n,1}^-\, e_n).
\end{cases}
$$
\item If $n\in\mathbb{E}_2$, we have exactly 4 nontrivial unimodal solutions of the form
$$
(u,v) = \begin{cases}
(\alpha_{n,1}^+\, e_n,\,\alpha_{n,1}^+\, e_n)\\
(\alpha_{n,1}^-\, e_n,\, \alpha_{n,1}^-\, e_n)\\
(\alpha_{n,2}^+\, e_n,\,\alpha_{n,2}^-\, e_n)\\
(\alpha_{n,2}^-\, e_n,\, \alpha_{n,2}^+\, e_n).
\end{cases}
$$
\item If $n\in\mathbb{E}_3$, we have exactly 8 nontrivial unimodal solutions of the form
$$
(u,v) =
\begin{cases}
(\alpha_{n,1}^+ \,e_n,\,\alpha_{n,1}^+ \,e_n)\\
(\alpha_{n,1}^- \,e_n,\, \alpha_{n,1}^- \,e_n)\\
(\alpha_{n,2}^+\, e_n,\,\alpha_{n,2}^-\, e_n)\\
(\alpha_{n,2}^-\, e_n,\, \alpha_{n,2}^+\, e_n)\\
(\alpha_{n,3}^+\, e_n,\,\alpha_{n,4}^-\, e_n)\\
(\alpha_{n,3}^-\, e_n,\, \alpha_{n,4}^+\, e_n)\\
(\alpha_{n,4}^+\, e_n,\,\alpha_{n,3}^-\, e_n)\\
(\alpha_{n,4}^-\, e_n,\, \alpha_{n,3}^+\, e_n).
\end{cases}
$$
\item If $n\notin\mathbb{E}$, all the unimodal solutions involving the eigenvector $e_{n}$ are trivial.
\end{itemize}
In summary, system \eqref{MAIN} admits $2|\mathbb{E}_1|+4|\mathbb{E}_2|+8|\mathbb{E}_3|$ nontrivial unimodal solutions.
\end{theorem}

\begin{proof}
Let us look for nontrivial solutions $(u,v)$ of the form \eqref{form}.
Choosing $\phi=\psi=e_n$ in the weak formulation \eqref{weak} and recalling \eqref{Bautc1}, we obtain the system
$$
\begin{cases}
\lambda_n^2 \alpha_n + (\beta + \cz\lambda_n\alpha_n^2)\lambda_n\alpha_n + k(\alpha_n-\gamma_n)=0,\\\noalign{\vskip0.3mm}
\lambda_n^2 \gamma_n + (\beta + \cz\lambda_n\gamma_n^2)\lambda_n\gamma_n - k(\alpha_n-\gamma_n)=0,
\end{cases}
$$
which, setting
$$
\eta_n = 1 +  \frac{\beta}{\lambda_n}+ \frac{k}{\lambda_n^2} \and \omega_n = \frac{\lambda_n^2}{k},
$$
can be rewritten as
\begin{equation}
\label{relation}
\begin{cases}
\gamma_n = \omega_n \alpha_n(\eta_n + \cz\alpha_n^2),\\
\alpha_n= \omega_n \gamma_n(\eta_n + \cz\gamma_n^2).
\end{cases}
\end{equation}
Solving with respect to $\alpha_n$, we arrive at the nine-order equation
$$
\alpha_n(\cz^4\alpha_n^8 \omega_n^4 + 3 \cz^3\alpha_n^6 \omega_n^4 \eta _n+ 3 \cz^2\alpha_n^4 \omega_n^4 \eta_n^2  +
\cz\alpha_n^2\omega_n^4\eta_n^3 + \cz\alpha_n^2\omega_n^2\eta_n +\omega_n^2\eta_n^2 - 1) = 0.
$$
If $\alpha_n=0$ the solution is trivial (since in this case also $\gamma_n$ is zero).
Otherwise, introducing the auxiliary variable
$$
x_n = \omega_n(\eta_n + \cz\alpha_n^2),
$$
we end up with
$$
(x_n^2 - 1 )(x_n^2 - x_n \omega_n\eta_n + 1) = 0.
$$
Making use of the relations
\begin{equation}
\label{pg}
\begin{cases}
\displaystyle
\omega_n \eta_n = -\frac{\lambda_n}{k}\big(-\beta + \mu_n -\nu_n- \lambda_n\big),\\\noalign{\vskip1.5mm}
\displaystyle
\omega_n^2 \eta_n^2 - 4 = \frac{\lambda_n^2}{k^2}\big(\beta + \lambda_n +\mu_n- \nu_n\big)(\beta+ \nu_n),
\end{cases}
\end{equation}
one can easily realize that the solutions are
the {\sc u}-amplitudes $\alpha_{n,i}^\pm$ given by \eqref{ALPHA}.
Hence, according to Lemma~\ref{reale}, we have exactly
\begin{itemize}
\item 2 distinct nontrivial solutions $\{\alpha_{n,1}^\pm\}$ for every $n\in\mathbb{E}_1$;
\item 4 distinct nontrivial solutions $\{\alpha_{n,1}^\pm,\alpha_{n,2}^\pm\}$ for every $n\in\mathbb{E}_2$;
\item 8 distinct nontrivial solutions $\{\alpha_{n,1}^\pm,\alpha_{n,2}^\pm,
\alpha_{n,3}^\pm,\alpha_{n,4}^\pm\}$ for every $n\in\mathbb{E}_3$.
\end{itemize}
By the same token, when $n\notin\mathbb{E}$, we have only the trivial solution.
We are left to find the explicit values $\gamma_{n,i}^\pm$, which can be obtained from \eqref{relation}.
To this end, it is apparent to see that
$$
\begin{cases}
\gamma_{n,1}^\pm= \alpha_{n,1}^\pm,\\\noalign{\vskip1mm}
\gamma_{n,2}^\pm= \alpha_{n,2}^\mp.
\end{cases}
$$
Moreover, invoking \eqref{pg} and observing that the product $\omega_n\eta_n$ is negative when $n\in\mathbb{E}_3$,
\begin{align*}
\gamma_{n,3}^\pm &= \pm \frac{\sqrt{k}\big(\omega_n\eta_n+\sqrt{\omega_n^2\eta_n^2 - 4}\big)}{2}
\sqrt{ \frac{-\omega_n\eta_n+\sqrt{\omega_n^2\eta_n^2 - 4}}{2\cz\lambda_n^2}} \\
&= \mp \sqrt{k} \sqrt{\frac{-\omega_n\eta_n-\sqrt{\omega_n^2\eta_n^2 - 4}}{2\cz\lambda_n^2}}
= \alpha_{n,4}^{\mp},
\end{align*}
and
\begin{align*}
\gamma_{n,4}^\pm= &= \pm \frac{\sqrt{k}\big(\omega_n\eta_n-\sqrt{\omega_n^2\eta_n^2 - 4}\big)}{2}
\sqrt{ \frac{-\omega_n\eta_n-\sqrt{\omega_n^2\eta_n^2 - 4}}{2\cz\lambda_n^2}} \\
&= \mp \sqrt{k} \sqrt{\frac{-\omega_n\eta_n+\sqrt{\omega_n^2\eta_n^2 - 4}}{2\cz\lambda_n^2}}
= \alpha_{n,3}^{\mp}.
\end{align*}
The theorem is proved.
\end{proof}
%%%%%%%%%%%%%%%%%%%%%%%%%%%%%%%%%%%%%%%%%%%%%%%%%

%%%%%%%%%%%%%%%%%%%%%%%%%%%%%%%%%%%%%%%%%%%%%%%%%
\section{Equidistributed Energy Solutions}
\label{EES}

\noindent
In order to investigate the existence of solutions to system~\eqref{MAIN} which are not necessarily unimodal,
we begin to analyze a particular but still very interesting situation.

\begin{definition}
A nontrivial solution $(u,v)$ is called an {\it equidistributed energy solution}
({\it {\sc ee}-solution} for brevity) if
\begin{equation}
\label{equality}
\|u\|_1=\|v\|_1
\quad \Leftrightarrow\quad C_u=C_v.
\end{equation}
\end{definition}

At first glance, this condition might look restrictive. Though, as we will see in the next two lemmas,
{\sc ee}-solutions are in fact quite general. In particular, they pop up whenever a mode of
$u$ is equal or opposite to the corresponding mode of $v$.

\begin{lemma}
\label{FILIPPO}
With reference to \eqref{espressione}, if
$$\alpha_{n_i}\alpha_{n_j}=\pm\gamma_{n_i}\gamma_{n_j}\not=0$$
for some (possibly coinciding) $n_i,n_j$, then $(u,v)$ is
an {\sc ee}-solution. In particular, this is the case when\footnote{In fact, we will implicitly show in our
analysis that the latter condition is necessary as well
in order to have {\sc ee}-solutions.}
$$
|\alpha_{n_i}|=|\gamma_{n_i}|\not=0
$$
for some $n_i$.
\end{lemma}

\begin{proof}
Let $n_i,n_j$ be such that
$$\alpha_{n_i}\alpha_{n_j}=\pm\gamma_{n_i}\gamma_{n_j}\not=0.$$
Choosing $\phi=\psi=e_{n_i}$ in the weak formulation \eqref{weak}, we obtain
\begin{equation}
\label{PRIMO}
\begin{cases}
\lambda_{n_i}^2 \alpha_{n_i} + C_u\lambda_{n_i}\alpha_{n_i}+k(\alpha_{n_i}-\gamma_{n_i})=0,\\
\lambda_{n_i}^2 \gamma_{n_i} + C_v\lambda_{n_i}\gamma_{n_i}-k(\alpha_{n_i}-\gamma_{n_i})=0,
\end{cases}
\end{equation}
while, choosing $\phi=\psi=e_{n_j}$, we get
\begin{equation}
\label{SECONDO}
\begin{cases}
\lambda_{n_j}^2 \alpha_{n_j} + C_u\lambda_{n_j}\alpha_{n_j}+k(\alpha_{n_j}-\gamma_{n_j})=0,\\
\lambda_{n_j}^2 \gamma_{n_j} + C_v\lambda_{n_j}\gamma_{n_j}-k(\alpha_{n_j}-\gamma_{n_j})=0.
\end{cases}
\end{equation}
Then, from \eqref{PRIMO},
$$
\begin{cases}
\displaystyle
C_u = -\lambda_{n_i} - \frac{k(\alpha_{n_i}-\gamma_{n_i})}{\lambda_{n_i}\alpha_{n_i}},\\\noalign{\vskip1.1mm}
\displaystyle
C_v = -\lambda_{n_i}+\frac{k(\alpha_{n_i}-\gamma_{n_i})}{\lambda_{n_i}\gamma_{n_i}}.
\end{cases}
$$
These expressions, substituted into \eqref{SECONDO}, yield
$$
\begin{cases}
\lambda_{n_j}^2\lambda_{n_i} \alpha_{n_i}\alpha_{n_j}-\lambda_{n_i}^2\lambda_{n_j} \alpha_{n_i}\alpha_{n_j}
- k\lambda_{n_j}\alpha_{n_j}(\alpha_{n_i}-\gamma_{n_i})+k\lambda_{n_i}\alpha_{n_i}(\alpha_{n_j}
-\gamma_{n_j})=0,\\\noalign{\vskip0.3mm}
\lambda_{n_j}^2\lambda_{n_i} \gamma_{n_i}\gamma_{n_j}-\lambda_{n_i}^2\lambda_{n_j} \gamma_{n_i}\gamma_{n_j}
+ k\lambda_{n_j}\gamma_{n_j}(\alpha_{n_i}-\gamma_{n_i})-k\lambda_{n_i}\gamma_{n_i}(\alpha_{n_j}-\gamma_{n_j})=0.
\end{cases}
$$
If
$$\alpha_{n_i}\alpha_{n_j}=\gamma_{n_i}\gamma_{n_j}\not=0,$$
subtracting the two equations of the system above we readily find
$$
|\alpha_{n_i}|=|\gamma_{n_i}|.
$$
On the other hand, if
$$\alpha_{n_i}\alpha_{n_j}=-\gamma_{n_i}\gamma_{n_j}\not=0,$$
(implying $n_i\not=n_j$), adding the two equations of the system we still conclude that
$$
|\alpha_{n_i}|=|\gamma_{n_i}|.
$$
At this point, an exploitation of \eqref{PRIMO} gives $C_u=C_v$.
\end{proof}

\begin{lemma}
\label{VITTO}
With reference to \eqref{espressione}, if
$$\alpha_{n_i}\gamma_{n_j}=\alpha_{n_j}\gamma_{n_i}\not=0$$
for some $n_i\not=n_j$, then $(u,v)$ is an {\sc ee}-solution.
\end{lemma}

\begin{proof}
By assumption, there exists $\varpi\not=0$ such that
$$
\alpha_{n_i}=\varpi \gamma_{n_i} \and \alpha_{n_j}=\varpi \gamma_{n_j}.
$$
Due to Lemma \ref{FILIPPO}, to reach the conclusion it is sufficient to show that $\varpi=-1$.
If not, exploiting \eqref{rel},
$$
\lambda_{n_i} = - \frac{C_u\alpha_{n_i}+C_v\gamma_{n_i}}{\alpha_{n_i}+\gamma_{n_i}}
= - \frac{C_u\varpi+C_v}{\varpi+1}=
- \frac{C_u\alpha_{n_j}+C_v\gamma_{n_j}}{\alpha_{n_j}+\gamma_{n_j}} = \lambda_{n_j},
$$
yielding a contradiction.
\end{proof}

We now proceed with a detailed description of the class of {\sc ee}-solutions.

\subsection{The unimodal case}
The unimodal solutions
have been already classified in the previous section. In particular, from Theorem \ref{unimodal}
we learn that all unimodal solutions, except the ones involving the {\sc u}-amplitudes $\alpha_{n,3}^\pm$
and $\alpha_{n,4}^\pm$ arising from the further bifurcation at $\nu_n=-\beta$, are in fact {\sc ee}-solutions. That is,
system \eqref{MAIN} admits
$$
2|\mathbb{E}_1| + 4|\mathbb{E}_2| + 4|\mathbb{E}_3|
$$
unimodal {\sc ee}-solutions, explicitly computed.

\subsection{The bimodal case}
In order to classify the bimodal {\sc ee}-solutions, we introduce the (disjoint and possibly empty)
subsets of $\mathbb{E}\times\mathbb{E}$
$$
\mathbb{B}_1 = \{(n_1,n_2): n_1<n_2,\, \lambda_{n_1} + \lambda_{n_2}<-\beta \,\text{ and }\, \lambda_{n_1}\lambda_{n_2}=2k \}
$$
and
$$
\mathbb{B}_2 = \{(n_1,n_2): n_1<n_2,\, \lambda_{n_2}<-\beta \,\text{ and }\,\lambda_{n_1}(\lambda_{n_2}-\lambda_{n_1})=2k\}.
$$
Then, setting
$$
\mathbb{B} = \mathbb{B}_1 \cup \mathbb{B}_2,
$$
we have the following result.

\begin{theorem}
\label{teobim}
System \eqref{MAIN} admits bimodal {\sc ee}-solutions if and only if the set $\mathbb{B}$ is nonempty. More
precisely, for every couple $(n_1,n_2)\in\N\times\N$ with $n_1<n_2$, one of the following disjoint situations occurs.
\begin{itemize}
\item If $(n_1,n_2)\in\mathbb{B}_1$, we have exactly the (infinitely many) solutions of the form
$$
\begin{cases}
u = x e_{n_1} + y e_{n_2},\\
v = -x e_{n_1} -y e_{n_2},
\end{cases}
$$
for all $(x,y)\in\R^2$ satisfying the equality
$$
\cz x^2 \lambda_{n_1} + \cz y^2\lambda_{n_2} + \lambda_{n_1}+\lambda_{n_2}+\beta =0 \qquad \text{with}\qquad xy\neq0.
$$
\item If $(n_1,n_2)\in\mathbb{B}_2$, we have exactly the (infinitely many) solutions of the form
$$
\begin{cases}
u = x e_{n_1} + y e_{n_2}, \\
v = -x e_{n_1} + y e_{n_2},
\end{cases}
$$
for all $(x,y)\in\R^2$ satisfying the equality
$$
\cz x^2 \lambda_{n_1} + \cz y^2\lambda_{n_2} + \lambda_{n_2}+\beta =0 \qquad \text{with}\qquad xy\neq0.
$$
\item If $(n_1,n_2)\notin\mathbb{B}$, there are no bimodal {\sc ee}-solutions involving the eigenvectors $e_{n_1}$
and $e_{n_2}$.
\end{itemize}
\end{theorem}

\begin{proof}
Let us look for bimodal {\sc ee}-solutions $(u,v)$ of the form
$$
\begin{cases}
u=\alpha_{n_1} e_{n_1} + \alpha_{n_2} e_{n_2},\\
v=\gamma_{n_1} e_{n_1} + \gamma_{n_2} e_{n_2},
\end{cases}
$$
with $n_1<n_2\in\N$ and $\alpha_{n_i},\gamma_{n_i}\in\R\setminus\{0\}$.
Choosing $\phi=\psi=e_{n_1}$ in the weak formulation~\eqref{weak}, we obtain
$$
\begin{cases}
\lambda_{n_1}^2 \alpha_{n_1} + C_u\lambda_{n_1}\alpha_{n_1}+k(\alpha_{n_1}-\gamma_{n_1})=0,\\
\lambda_{n_1}^2 \gamma_{n_1} + C_v\lambda_{n_1}\gamma_{n_1}-k(\alpha_{n_1}-\gamma_{n_1})=0,
\end{cases}
$$
while, choosing $\phi=\psi=e_{n_2}$, we get
$$
\begin{cases}
\lambda_{n_2}^2 \alpha_{n_2} + C_u\lambda_{n_2}\alpha_{n_2}+k(\alpha_{n_2}-\gamma_{n_2})=0,\\
\lambda_{n_2}^2 \gamma_{n_2} + C_v\lambda_{n_2}\gamma_{n_2}-k(\alpha_{n_2}-\gamma_{n_2})=0.
\end{cases}
$$
Since we require $C_u=C_v$, we infer that
\begin{align}
\label{EE1}
&C_u = -\lambda_{n_1}-\frac{k(\alpha_{n_1}-\gamma_{n_1})}{\lambda_{n_1}\alpha_{n_1}},\\
\label{EE2}
&C_u = -\lambda_{n_1}+\frac{k(\alpha_{n_1}-\gamma_{n_1})}{\lambda_{n_1}\gamma_{n_1}},\\
\label{EE3}
&C_u = -\lambda_{n_2}-\frac{k(\alpha_{n_2}-\gamma_{n_2})}{\lambda_{n_2}\alpha_{n_2}},\\
\label{EE4}
&C_u = -\lambda_{n_2}+\frac{k(\alpha_{n_2}-\gamma_{n_2})}{\lambda_{n_2}\gamma_{n_2}}.
\end{align}
At this point, we shall distinguish three cases.

\smallskip
\noindent
$\diamond$ When
$$
\begin{cases}
\gamma_{n_1} +\alpha_{n_1}=0,\\
\gamma_{n_2} +\alpha_{n_2}=0,
\end{cases}
$$
equations \eqref{EE1}-\eqref{EE4} reduce to
$$
\begin{cases}
\lambda_{n_1}C_u = -\lambda_{n_1}^2 - 2k,\\
\lambda_{n_2}C_u = -\lambda_{n_2}^2 - 2k,
\end{cases}
$$
implying
$$
\lambda_{n_1}\lambda_{n_2}=2k.
$$
Moreover, the value $C_u$ is determined by \eqref{Bautc1}, which provides the equality
$$
\cz \alpha_{n_1}^2\lambda_{n_1} + \cz \alpha_{n_2}^2\lambda_{n_2} + \lambda_{n_1}+\lambda_{n_2}+\beta=0.
$$
Hence, there exist bimodal {\sc ee}-solutions (explicitly computed) if and only if the pair $(n_1,n_2)\in\mathbb{B}_1$.

\smallskip
\noindent
$\diamond$ When
$$
\begin{cases}
\gamma_{n_1} +\alpha_{n_1}=0,\\
\gamma_{n_2} +\alpha_{n_2}\neq 0,
\end{cases}
$$
we take the difference of \eqref{EE4} and \eqref{EE3}, establishing the identity
$$
\gamma_{n_2}=\alpha_{n_2}.
$$
Thus, equations \eqref{EE1}-\eqref{EE4} reduce to
$$
\begin{cases}
\lambda_{n_1}C_u = -\lambda_{n_1}^2 - 2k,\\\noalign{\vskip1.3mm}
C_u = -\lambda_{n_2},
\end{cases}
$$
implying
$$
\lambda_{n_1}(\lambda_{n_2}-\lambda_{n_1})=2k.
$$
Again, the value $C_u$ is determined by \eqref{Bautc1}, which gives
$$
\cz \alpha_{n_1}^2\lambda_{n_1} + \cz \alpha_{n_2}^2\lambda_{n_2} + \lambda_{n_2}+\beta=0.
$$
Hence, there exist bimodal {\sc ee}-solutions (explicitly computed) if and only if the pair
$(n_1,n_2)\in\mathbb{B}_2$.

\smallskip
\noindent
$\diamond$ We show that the remaining case
$$\gamma_{n_1} +\alpha_{n_1}\neq 0$$
is impossible. Indeed, taking the difference of
\eqref{EE2} and \eqref{EE1}, we find
$$
\gamma_{n_1} =\alpha_{n_1}.
$$
If $\gamma_{n_2} +\alpha_{n_2}= 0$, from \eqref{EE1} and \eqref{EE3} we conclude that
$$
0<2k = \lambda_{n_2}(\lambda_{n_1}-\lambda_{n_2})<0,
$$
yielding a contradiction. On the other hand, if $\gamma_{n_2} +\alpha_{n_2}\neq 0$,
we learn once more that
$$\gamma_{n_2}=\alpha_{n_2}.$$
But in this situation, equations \eqref{EE1} and \eqref{EE3} lead to
$\lambda_{n_1}  = \lambda_{n_2},$
and the sought contradiction follows.
\end{proof}

\subsection{The trimodal case}
Finally, we classify the trimodal {\sc ee}-solutions. To
this end, we consider the (possibly empty)
subset of $\mathbb{E}\times\mathbb{E}\times\mathbb{E}$
$$
\mathbb{T} = \{(n_1,n_2,n_3): n_1<n_2<n_3,\, \lambda_{n_3}<-\beta \,\text{ and }\, \lambda_{n_1}(\lambda_{n_3}-\lambda_{n_1})
=\lambda_{n_2}(\lambda_{n_3}-\lambda_{n_2})=2k \}.
$$
The result reads as follows.

\begin{theorem}
\label{teotrimodal}
System \eqref{MAIN} admits trimodal {\sc ee}-solutions if and only if the set $\mathbb{T}$ is nonempty. More
precisely, for every triplet $(n_1,n_2,n_3)\in\N\times\N\times\N$ with $n_1<n_2<n_3$, one
of the following disjoint situations occurs.

\begin{itemize}
\item If $(n_1,n_2,n_3)\in\mathbb{T}$, we have exactly the (infinitely many) solutions of the form
$$
\begin{cases}
u = x e_{n_1} + y e_{n_2} + z e_{n_3},\\
v = -x e_{n_1} -y e_{n_2} + z e_{n_3},
\end{cases}
$$
for all $(x,y,z)\in\R^3$ satisfying the equality
$$
\cz x^2 \lambda_{n_1} + \cz y^2\lambda_{n_2} + \cz z^2\lambda_{n_3} + \lambda_{n_3}+\beta =0 \qquad \text{with}\qquad xyz\neq0.
$$
\item If $(n_1,n_2,n_3)\notin\mathbb{T}$, there are no trimodal {\sc ee}-solutions involving the eigenvectors
$e_{n_1}, e_{n_2}, e_{n_3}$.
\end{itemize}
\end{theorem}

\begin{proof}
The argument goes along the same lines of Theorem \ref{teobim}. For this reason, we limit ourselves
to give a short (albeit complete) proof, leaving the verification of some calculations to the reader.

As customary, let us look for trimodal {\sc ee}-solutions $(u,v)$ of the form
$$
\begin{cases}
u=\alpha_{n_1} e_{n_1} + \alpha_{n_2} e_{n_2} + \alpha_{n_3} e_{n_3},\\
v=\gamma_{n_1} e_{n_1} + \gamma_{n_2} e_{n_2} + \gamma_{n_3} e_{n_3},
\end{cases}
$$
with $n_1<n_2<n_3\in\N$ and $\alpha_{n_i},\gamma_{n_i}\in\R\setminus\{0\}$.
Accordingly, from the weak formulation~\eqref{weak}, choosing first $\phi=\psi=e_{n_1}$, then
$\phi=\psi=e_{n_2}$, and finally $\phi=\psi=e_{n_3}$, we obtain the six equations
\begin{equation}
\label{sistemone}
\begin{cases}
\displaystyle C_u = -\lambda_{n_1}-\frac{k(\alpha_{n_1}-\gamma_{n_1})}{\lambda_{n_1}\alpha_{n_1}},\\\noalign{\vskip1.3mm}
\displaystyle C_u = -\lambda_{n_1}+\frac{k(\alpha_{n_1}-\gamma_{n_1})}{\lambda_{n_1}\gamma_{n_1}},\\\noalign{\vskip1.3mm}
\displaystyle C_u = -\lambda_{n_2}-\frac{k(\alpha_{n_2}-\gamma_{n_2})}{\lambda_{n_2}\alpha_{n_2}},\\\noalign{\vskip1.3mm}
\displaystyle C_u = -\lambda_{n_2}+\frac{k(\alpha_{n_2}-\gamma_{n_2})}{\lambda_{n_2}\gamma_{n_2}},\\\noalign{\vskip1.3mm}
\displaystyle C_u = -\lambda_{n_3}-\frac{k(\alpha_{n_3}-\gamma_{n_3})}{\lambda_{n_3}\alpha_{n_3}},\\\noalign{\vskip1.3mm}
\displaystyle C_u = -\lambda_{n_3}+\frac{k(\alpha_{n_3}-\gamma_{n_3})}{\lambda_{n_3}\gamma_{n_3}},
\end{cases}
\end{equation}
where the condition
$C_u=C_v$ has been used.
The next step is to show that
\begin{equation}
\label{ctrim}
\begin{cases}
\gamma_{n_1}+\alpha_{n_1}=0,\\
\gamma_{n_2}+\alpha_{n_2}=0,\\
\gamma_{n_3}+\alpha_{n_3}\neq0,
\end{cases}
\end{equation}
being the remaining cases impossible. To prove the claim,
the argument is similar to the one of Theorem \ref{teobim}. For instance, assuming
$$
\begin{cases}
\gamma_{n_1}+\alpha_{n_1}=0,\\
\gamma_{n_2}+\alpha_{n_2}=0,\\
\gamma_{n_3}+\alpha_{n_3}=0,
\end{cases}
$$
system \eqref{sistemone} reduces to
$$
\begin{cases}
\lambda_{n_1}C_u = -\lambda_{n_1}^2 - 2k,\\\noalign{\vskip1mm}
\lambda_{n_2}C_u = -\lambda_{n_2}^2 - 2k,\\\noalign{\vskip1mm}
\lambda_{n_3}C_u = -\lambda_{n_3}^2 - 2k,
\end{cases}
$$
forcing
$$
2k = \lambda_{n_1}\lambda_{n_2} = \lambda_{n_2}\lambda_{n_3}
$$
and yielding a contradiction. The other cases can be carried out analogously;
the details are left to the reader.
Within \eqref{ctrim}, we take the difference of the last two equations
of \eqref{sistemone}, and we obtain
$$
\gamma_{n_3}=\alpha_{n_3}.
$$
Thus, system \eqref{sistemone} turns into
$$
\begin{cases}
\lambda_{n_1}C_u = -\lambda_{n_1}^2 - 2k,\\\noalign{\vskip1mm}
\lambda_{n_2}C_u = -\lambda_{n_2}^2 - 2k,\\\noalign{\vskip1mm}
C_u = -\lambda_{n_3},
\end{cases}
$$
implying
$$
\lambda_{n_1}(\lambda_{n_3}-\lambda_{n_1})
=\lambda_{n_2}(\lambda_{n_3}-\lambda_{n_2})=2k.
$$
Moreover, the value $C_u$ is determined by \eqref{Bautc1}, which provides the equality
$$
\cz \alpha_{n_1}^2\lambda_{n_1} + \cz \alpha_{n_2}^2\lambda_{n_2} + \cz \alpha_{n_3}^2 \lambda_{n_3} + \lambda_{n_3}+\beta=0.
$$
Hence, there exist trimodal {\sc ee}-solutions (explicitly computed) if and only if the triplet $(n_1,n_2,n_3)\in\mathbb{T}$.
\end{proof}

\begin{corollary}
\label{WWW}
Let $(u,v)$ be a trimodal {\sc ee}-solution. Then, with reference to \eqref{espressione},
if $n_1<n_2<n_3$
the eigenvalues $\lambda_{n_1},\lambda_{n_2},\lambda_{n_3}$
fulfill the relation
$$
\lambda_{n_1}+\lambda_{n_2}=\lambda_{n_3}.
$$
\end{corollary}

\begin{proof}
In the light of Theorem \ref{teotrimodal}, we know that $(n_1,n_2,n_3)\in\mathbb{T}$. In particular,
$$
\lambda_{n_1}(\lambda_{n_3}-\lambda_{n_1})=\lambda_{n_2}(\lambda_{n_3}-\lambda_{n_2}).
$$
Since $\lambda_{n_1}\neq\lambda_{n_2}$, the conclusion follows.
\end{proof}
%%%%%%%%%%%%%%%%%%%%%%%%%%%%%%%%%%%%%%%%%%%%%%%%%%

%%%%%%%%%%%%%%%%%%%%%%%%%%%%%%%%%%%%%%%%%%%%%%%%%%
\section{General Bimodal Solutions}
\label{GB}

\noindent
In this section, we investigate the existence of general (not necessarily equidistributed)
bimodal solutions to system \eqref{MAIN}.
First, specializing Lemmas \ref{FILIPPO} and \ref{VITTO}, we obtain

\begin{theorem}
\label{bim}
Let $(u,v)$ be a bimodal solution. With reference to \eqref{espressione}, if
\begin{itemize}
\item $|\alpha_{n_1}|=|\gamma_{n_1}|\not=0$, or
\smallskip
\item $|\alpha_{n_2}|=|\gamma_{n_2}|\not=0$, or
\smallskip
\item $\alpha_{n_1}\alpha_{n_2}=\pm \gamma_{n_1}\gamma_{n_2}\not=0$, or
\smallskip
\item $\alpha_{n_1}\gamma_{n_2}=\alpha_{n_2}\gamma_{n_1}\not=0$,
\end{itemize}
then $(u,v)$ is an {\sc ee}-solution.
\end{theorem}

Even if Theorem \ref{bim} somehow tells that a bimodal solution is likely to be
an {\sc ee}-solution, it is possible to have bimodal solutions of not
equidistributed energy. Indeed, the complete picture will be given in the next Theorem \ref{teobimgen}
of \S\ref{CGBS}. Some preparatory work is needed.

\subsection{Technical lemmas}
In what follows, $(n_1,n_2)\in\N\times\N$ is an arbitrary, but fixed,
pair of natural numbers, with $n_1<n_2$. We will introduce several quantities depending on $(n_1,n_2)$.
Setting
\begin{equation}
\label{zzz}
\zeta=\zeta(n_1,n_2)=\frac{\lambda_{n_2}}{\lambda_{n_1}}>1,
\end{equation}
and
\begin{equation}
\label{www}
\sigma=\sigma(n_1,n_2)=\frac{k-\lambda_{n_1}\lambda_{n_2}}{k}\in\R,
\end{equation}
we consider the real numbers (defined whenever $\sigma\neq0$)
$$
\Phi = \Phi(n_1,n_2) =\frac{(\zeta+1)+(\zeta-1)\sigma^2}{\sigma\zeta},
$$
and
$$
\Psi = \Psi(n_1,n_2) =\frac{(\zeta+1)-(\zeta-1)\sigma^2}{\sigma}.
$$
By direct computations, we have the identity
$$
\Phi^2\zeta^2 - \Psi^2 = 4(\zeta^2-1),
$$
which, in turn, yields
\begin{equation}
\label{PS}
(\Phi^2 - 4)\zeta^2=\Psi^2 - 4= \frac{(\zeta-1)^2\sigma^4
-2(\zeta^2 + 1)\sigma^2+(\zeta+1)^2}{\sigma^2 }.
\end{equation}
This relation will be useful later. Then, we introduce
the real numbers (whenever defined)
\begin{align*}
X = X(n_1,n_2)&= \frac{\Phi + \sqrt{\Phi^2-4}}{2},\\
Y = Y(n_1,n_2)&= \frac{\Phi - \sqrt{\Phi^2-4}}{2},\\
W = W(n_1,n_2)&= \frac{\Psi + \sqrt{\Psi^2-4}}{2},\\
Z = Z(n_1,n_2)&= \frac{\Psi - \sqrt{\Psi^2-4}}{2}.
\end{align*}

\begin{lemma}
\label{prelk}
The following are equivalent.
\begin{itemize}
\item At least one of the numbers $X,Y,W,Z$ belongs to $\R$.
\smallskip
\item All the numbers $X,Y,W,Z$ belong to $\R$.
\smallskip
\item $\lambda_{n_1}\lambda_{n_2}\in (0,2k]\setminus \{ k \}$
or $\lambda_{n_1}(\lambda_{n_2}-\lambda_{n_1})\in [2k,\infty)$.
\end{itemize}
\end{lemma}

\begin{proof}
It is apparent to see that
$$X\in\R \quad \Leftrightarrow\quad \Phi^2\geq 4 \quad \Leftrightarrow\quad  Y\in\R,$$
and
$$W\in\R \quad \Leftrightarrow\quad \Psi^2\geq 4 \quad \Leftrightarrow\quad  Z\in\R.$$
Moreover, in the light of \eqref{PS},
$$\Phi^2\geq 4 \quad \Leftrightarrow\quad \Psi^2\geq 4.
$$
Therefore, in order to reach the conclusion, it is sufficient to show that
$$
\Psi^2\geq 4 \quad \Leftrightarrow \quad
\lambda_{n_1}\lambda_{n_2}\in(0,2k]\setminus \{ k \}\quad \text{or}\quad
\lambda_{n_1}(\lambda_{n_2}-\lambda_{n_1})\in [2k,\infty).
$$
To this end, exploiting \eqref{PS},
$$
\Psi^2\geq 4 \quad \Leftrightarrow \quad
\begin{cases}
\lambda_{n_1}\lambda_{n_2}\neq k,\\\noalign{\vskip0.7mm}
(\zeta-1)^2\sigma^4 -2(\zeta^2+1)\sigma^2 + (\zeta+1)^2\geq0.
\end{cases}
$$
Making use of the trivial inequality $\sigma<1$,
one can verify by elementary calculations that
$$
(\zeta-1)^2\sigma^4 -2(\zeta^2+1)\sigma^2 + (\zeta+1)^2\geq0
$$
if and only if
$$
\sigma \in \big(-\infty,\frac{\zeta+1}{1-\zeta}\bigg] \cup [-1,1).
$$
Since
$$
\sigma \in \big(-\infty,\frac{\zeta+1}{1-\zeta}\bigg]
\quad \Leftrightarrow \quad
\lambda_{n_1}(\lambda_{n_2}-\lambda_{n_1})\in [2k,\infty),
$$
and
$$
\sigma \in [-1,1)
\quad \Leftrightarrow \quad
\lambda_{n_1}\lambda_{n_2}\in (0,2k] \setminus \{ k \},
$$
the proof is finished.
\end{proof}

\begin{lemma}
\label{bri}
The following are equivalent.
\begin{itemize}
\item $X=Y$.
\smallskip
\item $W=Z$.
\smallskip
\item $\lambda_{n_1}\lambda_{n_2}=2k$\, or\, $\lambda_{n_1}(\lambda_{n_2}-\lambda_{n_1}) = 2k$.
\end{itemize}
\end{lemma}

The argument goes along the same lines of Lemma \ref{prelk} (actually, it is even simpler). For this reason,
the proof is omitted and left to the reader.

\smallskip
At this point, we state a simple but crucial identity, which follows immediately from \eqref{PS} and
the definitions of the numbers $\zeta,\Phi,\Psi,X,Y,W,Z$.

\begin{lemma}
We have the equality
\begin{equation}
\label{g1}
\zeta X -W = \zeta Y - Z=(\zeta-1)\sigma,
\end{equation}
provided that the expressions above are well-defined.
\end{lemma}

\subsection{The numbers $\mathfrak{m}$ and $\mathfrak{M}$}
A crucial role in our analysis will be played by the following
two real numbers (again, defined whenever $\sigma\neq0$)
\begin{equation}
\label{mpiccolo}
\mathfrak{m}= \mathfrak{m}(n_1,n_2) =  \frac{k^2 + k \lambda_{n_2}(\lambda_{n_2}-\lambda_{n_1})
+\lambda_{n_1}^2\lambda_{n_2}^2}{(\lambda_{n_1}\lambda_{n_2}-k)\lambda_{n_2}},
\end{equation}
and
\begin{equation}
\label{mgrande}
\mathfrak{M} = \mathfrak{M}(n_1,n_2) =\frac{k^2 - k \lambda_{n_1}(\lambda_{n_2}-\lambda_{n_1})
+\lambda_{n_1}^2\lambda_{n_2}^2}{(\lambda_{n_1}\lambda_{n_2}-k)\lambda_{n_1}}.
\end{equation}
In particular, it is immediate to verify that
$$
\sigma<0 \quad\Rightarrow \quad \mathfrak{M}  > \mathfrak{m} > 0.
$$
Such numbers can be written in several different ways as functions of $X,Y,W,Z$.
To see that, we will exploit the relations
\begin{align}
\label{key1}
\begin{cases}
XY = 1,\\
X+Y = \Phi,\\
WZ = 1,\\
W+Z = \Psi,
\end{cases}
\end{align}
valid whenever $X,Y,W,Z\in\R$.
Then, setting
\begin{align*}
&f=f(n_1,n_2)=\frac{kX - \lambda_{n_1}^2-k}{\lambda_{n_1}},\\\noalign{\vskip1mm}
&g=g(n_1,n_2)=\frac{kY - \lambda_{n_1}^2-k}{\lambda_{n_1}},
\end{align*}
and making use of \eqref{g1}, it is easy to prove that
\begin{equation}
\label{USO}
\begin{cases}
\displaystyle
f = \frac{kW - \lambda_{n_2}^2-k}{\lambda_{n_2}},\\\noalign{\vskip2mm}
\displaystyle
g = \frac{kZ - \lambda_{n_2}^2-k}{\lambda_{n_2}}.
\end{cases}
\end{equation}

\begin{lemma}
\label{lll}
We have the equalities
$$
\mathfrak{m} = -g - \frac{k W^2 (X-Y)}{\lambda_{n_1}(W^2-1)}= - g - \frac{k (X-Y)}{\lambda_{n_1}(1-Z^2)},
$$
and
$$
\mathfrak{M} = -g - \frac{kX^2(X-Y)}{\lambda_{n_1}(X^2-1)}= - g - \frac{k(X-Y)}{\lambda_{n_1}(1-Y^2)}.
$$
provided that the expressions above are well-defined.
\end{lemma}

\begin{proof}
Exploiting \eqref{key1}, we obtain the identities
\begin{align*}
&\frac{W^2}{W^2-1} = \frac{W}{W-Z}=\frac{1}{1-Z^2},\\\noalign{\vskip1mm}
&\frac{X^2}{X^2-1} = \frac{X}{X-Y}=\frac{1}{1-Y^2}.
\end{align*}
Thus, in order to complete the proof, it is sufficient to show that
$$
-\mathfrak{m} = g + \frac{k W^2 (X-Y)}{\lambda_{n_1}(W^2-1)},
$$
and
$$
-\mathfrak{M} = g + \frac{kX^2(X-Y)}{\lambda_{n_1}(X^2-1)}.
$$
To this end,
in the light of \eqref{g1}, \eqref{key1},
\eqref{USO} and the definitions of $\zeta,\sigma,\Psi,g$, we compute
\begin{align*}
g +\frac{k W^2 (X-Y)}{\lambda_{n_1}(W^2-1)} &=\frac{kY - \lambda_{n_1}^2-k}{\lambda_{n_1}}+
\frac{k W^2 (X-Y)}{\lambda_{n_1}(W^2-1)} \\\noalign{\vskip0.7mm}
&=\frac{kZ - \lambda_{n_2}^2-k}{\lambda_{n_2}}+
\frac{k W^2 (W-Z)}{\lambda_{n_2}(W^2-1)} \\\noalign{\vskip0.7mm}
&=\frac{kZ - \lambda_{n_2}^2-k}{\lambda_{n_2}}+
\frac{k W}{\lambda_{n_2}} \\\noalign{\vskip0.7mm}
&= \frac{k\Psi- \lambda_{n_2}^2-k}{\lambda_{n_2}}\\\noalign{\vskip0.7mm}
&= \frac{k\zeta - k\sigma^2\zeta +k\sigma^2
-\sigma \lambda_{n_2}^2 + \lambda_{n_1}\lambda_{n_2}}{\sigma \lambda_{n_2}} \\\noalign{\vskip0.7mm}
&= \frac{(k-\lambda_{n_1}\lambda_{n_2})^2 + k\lambda_{n_2}^2
+k\lambda_{n_1}\lambda_{n_2}}{\sigma k\lambda_{n_2}} \\\noalign{\vskip0.7mm}
&=  -\mathfrak{m},
\end{align*}
while, making use of \eqref{key1}, along with the definitions of $\zeta,\sigma,\Phi,g$, we have
\begin{align*}
g + \frac{kX^2(X-Y)}{\lambda_{n_1}(X^2-1)} &=\frac{kY - \lambda_{n_1}^2-k}{\lambda_{n_1}}+ \frac{kX}{\lambda_{n_1}}
\\\noalign{\vskip0.7mm}
&=\frac{k\Phi- \lambda_{n_1}^2-k}{\lambda_{n_1}}\\\noalign{\vskip0.7mm}
&= \frac{k + k\sigma^2\zeta -k\sigma^2
-\sigma \lambda_{n_1}\lambda_{n_2} + \lambda_{n_2}^2}{\sigma \lambda_{n_1}\zeta} \\\noalign{\vskip0.7mm}
&= \frac{(k-\lambda_{n_1}\lambda_{n_2})^2 + k\lambda_{n_1}^2
+k\lambda_{n_2}\lambda_{n_1}}{\sigma k\lambda_{n_1}} \\\noalign{\vskip0.7mm}
&=  -\mathfrak{M}.
\end{align*}
The lemma is proved.
\end{proof}

\subsection{The circle-ellipse systems}

We need to investigate the solvability of the circle-ellipse systems
\begin{equation}
\label{SIS1}
\begin{cases}
\cz r^2 \lambda_{n_1} + \cz t^2\lambda_{n_2} + \beta = f,\\\noalign{\vskip1mm}
\cz r^2 \lambda_{n_1}X^2 + \cz t^2 \lambda_{n_2}W^2 + \beta = g,
\end{cases}
\end{equation}
and
\begin{equation}
\label{SIS2}
\begin{cases}
\cz r^2 \lambda_{n_1} + \cz t^2\lambda_{n_2} + \beta = g,\\\noalign{\vskip1mm}
\cz r^2 \lambda_{n_1}Y^2 + \cz t^2 \lambda_{n_2}Z^2 + \beta = f,
\end{cases}
\end{equation}
in the unknowns $r$ and $t$.

\begin{lemma}
\label{etabeta}
The following hold.
\begin{itemize}
\item Let $\lambda_{n_1}\lambda_{n_2}\in(0,k)$. Then neither
system \eqref{SIS1} nor
\eqref{SIS2} admit real solutions.

\smallskip
\item Let $\lambda_{n_1}\lambda_{n_2}\in(k,2k)$.
Then system \eqref{SIS1}
admits real solutions $(r,t)$ with $rt\neq0$ if and only if the same does
\eqref{SIS2}, if and only if
$$\mathfrak{m} < -\beta < \mathfrak{M}.$$
In which case, system \eqref{SIS1} admits exactly four distinct real solutions, and the
same does \eqref{SIS2}. Besides, they do not share any solution.

\smallskip
\item Let $\lambda_{n_1}(\lambda_{n_2}-\lambda_{n_1})\in(2k,\infty)$. Then system \eqref{SIS1}
admits real solutions $(r,t)$ with $rt\neq0$ if and only if the same does
\eqref{SIS2}, if and only if
$$
\mathfrak{M}<-\beta.
$$
In which case, system \eqref{SIS1} admits exactly four distinct real solutions, and the
same does \eqref{SIS2}. Besides, they do not share any solution.
\end{itemize}
\end{lemma}

\begin{proof}
We first observe that systems \eqref{SIS1} and \eqref{SIS2}
do not share any solution. Indeed, if it were so, we would have $f=g$ (meaning that $X=Y$) and
therefore, in the light of Lemma \ref{bri},
$$
\lambda_{n_1}\lambda_{n_2} = 2k {\qquad\text{or}\qquad}
\lambda_{n_1}(\lambda_{n_2}-\lambda_{n_1})=2k.
$$
Then, setting $s=\sqrt{\zeta}t$, we can rewrite \eqref{SIS1} and \eqref{SIS2} as
\begin{equation}
\label{SIS1bis}
\begin{cases}
r^2 + s^2 = F,\\\noalign{\vskip1mm}
X^2 r^2 + W^2 s^2 = G,
\end{cases}
\end{equation}
and
\begin{equation}
\label{SIS2bis}
\begin{cases}
r^2 + s^2= G,\\\noalign{\vskip1mm}
Y^2r^2 + Z^2s^2 = F,
\end{cases}
\end{equation}
where
$$
F= \frac{f-\beta}{\cz\lambda_{n_1}}
\and
G= \frac{g-\beta}{\cz\lambda_{n_1}}.
$$
In particular, calling
$$\nu= \frac{k (X-Y)}{\cz\lambda_{n_1}^2}\geq0,$$
we have the equality
\begin{equation}
\label{FG}
F = G + \nu.
\end{equation}
Systems \eqref{SIS1bis} and \eqref{SIS2bis} represent the intersection between a circle and an ellipse,
both centered at the origin. Therefore, real solutions $(r,s)$ with $rs\neq0$ exist if and only if the radius
of the circle is strictly greater than the minor semi-axis of the ellipse and
strictly smaller than the major semi-axis of the ellipse. In such a case, there are exactly four distinct solutions.
We shall distinguish three cases.

\medskip
\noindent
$\diamond$ {\it Case 1:\ }$\lambda_{n_1}\lambda_{n_2}\in (0,k)$.
By direct computations, one can easily see that
$$
\Psi > \Phi > 2,
$$
implying
$$
W>X>1>Y>Z>0.
$$
In particular, the number $\nu$ is strictly positive. As a consequence,
in the light of the discussion above and \eqref{FG}, system~\eqref{SIS1bis}
admits real solutions $(r,s)$ with $rs\neq0$ if and only if
$$
\frac{G}{W^2} < G+ \nu < \frac{G}{X^2}.
$$
Being $X^2>1$, it is apparent to see that the relation above is impossible.
Analogously, system \eqref{SIS2bis}
admits real solutions $(r,s)$ with $rs\neq0$ if and only if
$$
\frac{G+\nu}{Y^2} < G  < \frac{G+\nu}{Z^2}.
$$
Again, being $Y^2<1$, the relation is impossible.
In conclusion, neither
system~\eqref{SIS1bis} nor \eqref{SIS2bis} admit real solutions.

\medskip
\noindent
$\diamond$ {\it Case 2:\ }$\lambda_{n_1}\lambda_{n_2}\in (k,2k)$.
By direct computations, one can easily see that
$$
\Psi < \Phi < -2,
$$
implying
$$
Z < Y < -1 < X < W < 0.
$$
Analogously to the previous case, we infer that system \eqref{SIS1bis}
admits real solutions $(r,s)$ with $rs\neq0$ if and only if
$$
\frac{G}{X^2} < G+ \nu < \frac{G}{W^2}.
$$
Being $W^2<1$ and $X^2<1$, in the light of Lemma \ref{lll} we get
$$
\mathfrak{m} = - g - \frac{k W^2 (X-Y)}{\lambda_{n_1}(W^2-1)} < -\beta < -g - \frac{k X^2 (X-Y)}{\lambda_{n_1}(X^2-1)}=\mathfrak{M}.
$$
Moreover, system \eqref{SIS2bis}
admits real solutions $(r,s)$ with $rs\neq0$ if and only if
$$
\frac{G+\nu}{Z^2} < G  < \frac{G+\nu}{Y^2}.
$$
Being $Z^2>1$ and $Y^2>1$, invoking Lemma \ref{lll} we conclude that
$$
\mathfrak{m} = -g - \frac{k (X-Y)}{\lambda_{n_1}(1-Z^2)}< -\beta < -g - \frac{k (X-Y)}{\lambda_{n_1}(1-Y^2)}=\mathfrak{M}.
$$

\medskip
\noindent
$\diamond$ {\it Case 3:\ }$\lambda_{n_1}(\lambda_{n_2}-\lambda_{n_1})\in (2k,\infty)$.
By direct computations, one can easily see that
$$
\Phi < - 2 \and \Psi > 2,
$$
implying
$$
Y < -1 < X < 0 < Z < 1 < W.
$$
Arguing as in the previous cases, system~\eqref{SIS1bis}
admits real solutions $(r,s)$ with $rs\neq0$ if and only if
$$
\frac{G}{W^2} < G+ \nu < \frac{G}{X^2}.
$$
Since $W^2>1$, the relation above reduces to
$$
G+ \nu < \frac{G}{X^2}.
$$
Being $X^2<1$, making use of Lemma \ref{lll} we end up with
$$
\mathfrak{M}=- g - \frac{kX^2(X-Y)}{\lambda_{n_1}(X^2-1)}<-\beta.
$$
On the other hand, system \eqref{SIS2bis}
admits real solutions $(r,s)$ with $rs\neq0$ if and only if
$$
\frac{G+\nu}{Y^2} < G  < \frac{G+\nu}{Z^2}.
$$
Again, since $0 < Z^2 < 1$, the relation above reduces to
$$
\frac{G+\nu}{Y^2} < G.
$$
Being $Y^2>1$, an exploitation of Lemma \ref{lll} leads to
$$
\mathfrak{M}= - g - \frac{k(X-Y)}{\lambda_{n_1}(1-Y^2)}<-\beta.
$$
The proof is finished.
\end{proof}

\subsection{Classification of general bimodal solutions}
\label{CGBS}
In order to classify the general bimodal solutions, we introduce the
(disjoint and possibly empty) subsets of $\N \times \N$, with
$\mathfrak{m}$ and $\mathfrak{M}$ given by \eqref{mpiccolo} and \eqref{mgrande},
$$
\mathbb{B}^\star_1 = \big\{(n_1,n_2): n_1<n_2,\,\mathfrak{m}<-\beta<\mathfrak{M} \,\text{ and }\,
 \lambda_{n_1}\lambda_{n_2} \in(k,2k)  \big\},
$$
and
$$
\mathbb{B}^\star_2 = \big\{(n_1,n_2): n_1<n_2,\, \mathfrak{M}<-\beta \,\text{ and }\,
\lambda_{n_1}(\lambda_{n_2}-\lambda_{n_1})\in(2k,\infty)\big\},
$$
and we set
$$\mathbb{B}^\star=\mathbb{B}^\star_1\cup \mathbb{B}^\star_2.$$

\begin{lemma}
We have the inclusion $\mathbb{B}^\star\subset \mathbb{E} \times \mathbb{E}$.
In particular,
$\mathbb{B}^\star$ has finite cardinality.
\end{lemma}

\begin{proof}
By means of elementary computations, one can easily verify that the following implications hold:
\begin{align*}
&\lambda_{n_1} \lambda_{n_2} \in (k,2k)  \quad\Rightarrow \quad \lambda_{n_2}<\mathfrak{m} ,\\
&\lambda_{n_1}(\lambda_{n_2}-\lambda_{n_1}) \in(2k,\infty) \quad\Rightarrow \quad \lambda_{n_2}< \mathfrak{M}.
\end{align*}
Therefore, by the very definitions of
$\mathbb{B}^\star$ and $\mathbb{E}$,
$$
(n_1,n_2)\in\mathbb{B}^\star \quad\Rightarrow \quad
(n_1,n_2)\in \mathbb{E}\times \mathbb{E},
$$
as claimed.
\end{proof}

We have now all the ingredients to state our main theorem.

\begin{theorem}
\label{teobimgen}
System \eqref{MAIN} admits bimodal solutions of not
equidistributed energy if and only if the set $\mathbb{B}^\star$ is nonempty.
More precisely, for every couple $(n_1,n_2)\in \N\times\N$ with $n_1<n_2$,
one of the following disjoint situations occurs.
\begin{itemize}
\item If $(n_1,n_2)\in\mathbb{B}^\star$, we have exactly 8 distinct
bimodal solutions of not
equidistributed energy: 4 of the form
$$
\begin{cases}
u = r e_{n_1} + t e_{n_2},\\
v = r X e_{n_1} + t W e_{n_2},
\end{cases}
$$
where $r,t$ solve system \eqref{SIS1},
and 4 of the form
$$
\begin{cases}
u = r e_{n_1} + t e_{n_2},\\
v = r Y e_{n_1} + t Z e_{n_2},
\end{cases}
$$
where $r,t$ solve system \eqref{SIS2}.
\item If $(n_1,n_2)\notin\mathbb{B}^\star$, there are no bimodal solutions of not
equidistributed energy
involving the eigenvectors $e_{n_1}$
and $e_{n_2}$.
\end{itemize}
In summary, system~\eqref{MAIN} admits
$8|\mathbb{B}^\star|$ bimodal solutions of not
equidistributed energy.
\end{theorem}

\begin{proof}
Let us look for bimodal solutions of not equidistributed energy $(u,v)$ of the form
$$
\begin{cases}
u = \alpha_{n_1} e_{n_1} + \alpha_{n_2} e_{n_2},\\
v = \gamma_{n_1} e_{n_1} + \gamma_{n_2} e_{n_2},
\end{cases}
$$
with $n_1<n_2 \in\N$ and $\alpha_{n_i},\gamma_{n_i}\in\R \setminus \{ 0 \}$.

\medskip
\noindent
$\diamond$ {\it Step 1.}
We preliminarily show that
\begin{equation}
\label{CONDkappa}
\lambda_{n_1}\lambda_{n_2}\in (0,2k)\setminus \{ k \}\quad \text{or}\quad
\lambda_{n_1}(\lambda_{n_2}-\lambda_{n_1})\in (2k,\infty).
\end{equation}
To this end,
with reference to the weak formulation \eqref{weak}, choosing first $\phi=\psi = e_{n_1}$ and then
$\phi=\psi = e_{n_2}$, we obtain the system
\begin{equation}
\label{cx}
\begin{cases}
\alpha_{n_1}(\lambda_{n_1}^2 + C_u\lambda_{n_1}+k)=k\gamma_{n_1},\\
\gamma_{n_1}(\lambda_{n_1}^2 + C_v\lambda_{n_1}+k)=k\alpha_{n_1},\\
\alpha_{n_2}(\lambda_{n_2}^2 + C_u\lambda_{n_2}+k)=k\gamma_{n_2},\\
\gamma_{n_2}(\lambda_{n_2}^2 + C_v\lambda_{n_2}+k)=k\alpha_{n_2}.
\end{cases}
\end{equation}
Next, setting
\begin{equation}
\label{cambiovar}
\begin{cases}
\displaystyle
x_{n_1}=\frac{\lambda_{n_1}^2 + C_u\lambda_{n_1}+k}{k},\\\noalign{\vskip1.7mm}
\displaystyle
y_{n_1}=\frac{\lambda_{n_1}^2 + C_v\lambda_{n_1}+k}{k},\\\noalign{\vskip1.7mm}
\displaystyle
x_{n_2}=\frac{\lambda_{n_2}^2 + C_u\lambda_{n_2}+k}{k},\\\noalign{\vskip1.7mm}
\displaystyle
y_{n_2}=\frac{\lambda_{n_2}^2 + C_v\lambda_{n_2}+k}{k},
\end{cases}
\end{equation}
we get
$$
\begin{cases}
x_{n_1}y_{n_1}=1,\\\noalign{\vskip0.7mm}
x_{n_2}y_{n_2}=1,\\\noalign{\vskip0.7mm}
\zeta x_{n_1} -(\zeta-1)\sigma=x_{n_2},\\\noalign{\vskip0.7mm}
\zeta y_{n_1} - (\zeta-1)\sigma=y_{n_2}.
\end{cases}
$$
Observe that $\sigma\neq0$, otherwise
$$
\begin{cases}
x_{n_1}y_{n_1}=1,\\\noalign{\vskip0.7mm}
x_{n_2}y_{n_2}=1,\\\noalign{\vskip0.7mm}
\zeta x_{n_1}=x_{n_2},\\\noalign{\vskip0.7mm}
\zeta y_{n_1}=y_{n_2},
\end{cases}
$$
yielding $\zeta^2=1$ and contradicting the assumption $n_1<n_2$.
Therefore, we obtain
\begin{align}
\label{L1}
&x_{n_1}y_{n_1}=1,\\
\label{L2}
&x_{n_1} + y_{n_1}=\Phi,\\
\label{L3}
&x_{n_2}y_{n_2}=1,\\
\label{L4}
&x_{n_2} + y_{n_2}=\Psi.
\end{align}
Clearly, the solutions are given by the four quadruplets
\begin{align*}
&(X,Y,W,Z),\\
&(X,Y,Z,W),\\
&(Y,X,W,Z),\\
&(Y,X,Z,W).
\end{align*}
Since at least one (hence all) of the quadruplets has to have real components, making use of Lemma \ref{prelk}
we infer that
$$
\lambda_{n_1}\lambda_{n_2}\in (0,2k]\setminus \{ k \}\quad \text{or}\quad
\lambda_{n_1}(\lambda_{n_2}-\lambda_{n_1})\in [2k,\infty).
$$
In addition, due to the fact that $(u,v)$ does not have equidistributed energy,
$$
C_u\neq C_v\quad \Rightarrow\quad x_{n_1}\neq y_{n_1}.
$$
Thus, an exploitation
of Lemma~\ref{bri} yields
$$
\begin{cases}
\lambda_{n_1}\lambda_{n_2}\neq 2k,\\
\lambda_{n_1}(\lambda_{n_2}-\lambda_{n_1})\neq 2k,
\end{cases}
$$
and \eqref{CONDkappa} follows.

\medskip
\noindent
$\diamond$ {\it Step 2.} We now prove that, within \eqref{CONDkappa},
the coefficients $\alpha_{n_1}$ and $\alpha_{n_2}$
are solutions of system \eqref{SIS1} or \eqref{SIS2}.
Indeed, from \eqref{cambiovar} and recalling the definitions of $f$ and $g$, four possibilities occur:
\begin{equation}
\label{E1}
\begin{cases}
\displaystyle
C_u = f=\frac{kW - \lambda_{n_2}^2-k}{\lambda_{n_2}},\\\noalign{\vskip2mm}
\displaystyle
C_v = g=\frac{kZ - \lambda_{n_2}^2-k}{\lambda_{n_2}},
\end{cases}
\end{equation}
or
\begin{equation}
\label{E2}
\begin{cases}
\displaystyle
C_u = f=\frac{kZ - \lambda_{n_2}^2-k}{\lambda_{n_2}},\\\noalign{\vskip2mm}
\displaystyle
C_v = g=\frac{kW - \lambda_{n_2}^2-k}{\lambda_{n_2}},
\end{cases}
\end{equation}
or
\begin{equation}
\label{E3}
\begin{cases}
\displaystyle
C_u = g=\frac{kW - \lambda_{n_2}^2-k}{\lambda_{n_2}},\\\noalign{\vskip2mm}
\displaystyle
C_v = f=\frac{kZ - \lambda_{n_2}^2-k}{\lambda_{n_2}},
\end{cases}
\end{equation}
or
\begin{equation}
\label{E4}
\begin{cases}
\displaystyle
C_u = g=\frac{kZ - \lambda_{n_2}^2-k}{\lambda_{n_2}},\\\noalign{\vskip2mm}
\displaystyle
C_v = f = \frac{kW - \lambda_{n_2}^2-k}{\lambda_{n_2}}.
\end{cases}
\end{equation}
At this point, exploiting \eqref{CONDkappa} and Lemma \ref{bri}, we learn that $W\neq Z$. As a consequence,
taking into account \eqref{USO}, we conclude that only systems \eqref{E1} and \eqref{E4} survive.
Recalling the explicit forms of $C_u$ and $C_v$ given by \eqref{Bautc1}, we remain with
$$
\begin{cases}
\cz\alpha_{n_1}^2\lambda_{n_1} + \cz\alpha_{n_2}^2\lambda_{n_2}+\beta =  f,\\
\cz\gamma_{n_1}^2 \lambda_{n_1} + \cz\gamma_{n_2}^2\lambda_{n_2} + \beta = g,
\end{cases}
$$
and
$$
\begin{cases}
\cz\alpha_{n_1}^2\lambda_{n_1} + \cz\alpha_{n_2}^2\lambda_{n_2} +\beta= g,\\
\cz\gamma_{n_1}^2 \lambda_{n_1} + \cz\gamma_{n_2}^2\lambda_{n_2} + \beta = f.
\end{cases}
$$
Finally, due to \eqref{cx}, in the first case we infer that
$$
\begin{cases}
\gamma_{n_1} = X \alpha_{n_1},\\
\gamma_{n_2} = W \alpha_{n_2},
\end{cases}
$$
while in the second one
$$
\begin{cases}
\gamma_{n_1} = Y \alpha_{n_1},\\
\gamma_{n_2} = Z \alpha_{n_2}.
\end{cases}
$$

\medskip
\noindent
$\diamond$ {\it Step 3.} Collecting Steps 1-2 and Lemma \ref{etabeta},
there exist bimodal solutions of not equidistributed energy (explicitly computed) if and only if the couple
$(n_1,n_2)\in\mathbb{B}^\star$.
\end{proof}

\subsection{Two explicit examples}
We conclude by showing two explicit examples of bimodal solutions of not
equidistributed energy. In what follows, in order to avoid the presence of unnecessary constants,
we take for simplicity $\cz=1$, and we choose
$$A=\frac1{\pi^2}L,$$
being $L$ the Laplace-Dirichlet
operator of the concrete Example~\ref{dlo}. Accordingly, the eigenvalues of $A$ read
$$\lambda_n=n^2,$$
with corresponding eigenvectors
$$e_n(x)=\sqrt{2}\,\sin (n\pi x).$$

\begin{example}
Let
$$k=3 \and (n_1,n_2)=(1,2).$$
In this situation, an easy computation shows that
\begin{align*}
X &= -2 + \sqrt{3},\\
Y &= -2 - \sqrt{3},\\
W &= -7 + 4\sqrt{3},\\
Z &= -7-4\sqrt{3},
\end{align*}
and
$$\mathfrak{m}=\frac{61}{4} <16=\mathfrak{M}.$$
Accordingly, if $\beta$ is such that
$$\frac{61}{4}<-\beta<16,$$
the couple $(n_1,n_2)$ belongs to $\mathbb{B}^\star_1$.
Hence, there exist four solutions of the form
$$
\begin{cases}
u = \alpha_{1} e_{1} + \alpha_{2} e_{2},\\\noalign{\vskip0.7mm}
v = (\sqrt{3}-2) \alpha_{1} e_{1} + (4\sqrt{3}-7) \alpha_{2} e_{2},
\end{cases}
$$
where $\alpha_{1},\alpha_{2}\in\R$ solve
the system
\begin{equation}
\label{Aessis}
\begin{cases}
\displaystyle
\alpha_1^2 + 4\alpha_2^2 = 3\sqrt{3} - 10- \beta,\\\noalign{\vskip1mm}
\displaystyle
\alpha_1^2 (\sqrt{3}-2)^2 + 4\alpha_2^2 (4\sqrt{3}-7)^2 = -3\sqrt{3} - 10- \beta,
\end{cases}
\end{equation}
and four solutions of the form
$$
\begin{cases}
u = \alpha_{1} e_{1} + \alpha_{2} e_{2},\\\noalign{\vskip1mm}
v = -(\sqrt{3}+2) \alpha_{1} e_{1} - (4\sqrt{3}+7) \alpha_{2} e_{2},
\end{cases}
$$
where $\alpha_{1},\alpha_{2}\in\R$ solve the system
\begin{equation}
\label{ABessis}
\begin{cases}
\displaystyle
\alpha_1^2 + 4\alpha_2^2 = -3\sqrt{3} - 10- \beta,\\\noalign{\vskip1mm}
\displaystyle
\alpha_1^2 (\sqrt{3}+2)^2 + 4\alpha_2^2 (4\sqrt{3}+7)^2 = 3\sqrt{3} - 10- \beta.
\end{cases}
\end{equation}
For instance, when $\beta=-31/2$, the solutions
of system~\eqref{Aessis} are
$$
(\pm \alpha_1,\pm\alpha_2)\and(\pm\alpha_1,\mp\alpha_2),
$$
with
\begin{align*}
&\alpha_1= -\sqrt{\frac{7\sqrt{3}-12}{26\sqrt{3}-45}}\approx-1.93185,\\\noalign{\vskip2mm}
&\alpha_2= -\frac12\sqrt{\frac{362\sqrt{3}-627}{2(5042\sqrt{3}-8733)}}\approx-1.31948,
\end{align*}
while the solutions
of system~\eqref{ABessis} are
$$
(\pm \alpha_1,\pm\alpha_2)\and(\pm\alpha_1,\mp\alpha_2),
$$
with
\begin{align*}
&\alpha_1= -\sqrt{\frac{7\sqrt{3}+12}{26\sqrt{3}+45}}\approx-0.51763,\\\noalign{\vskip2mm}
&\alpha_2= -\frac12\sqrt{\frac{362\sqrt{3}+627}{2(5042\sqrt{3}+8733)}}\approx-0.09473.
\end{align*}
\end{example}

\begin{example}
Let
$$k=1 \and (n_1,n_2)=(1,2).$$
In this situation, an easy computation shows that
\begin{align*}
X &= \frac{-4+\sqrt{7}}{3},\\
Y &= \frac{-4-\sqrt{7}}{3},\\
W &= \frac{11+4\sqrt{7}}{3},\\
Z &= \frac{11-4\sqrt{7}}{3},
\end{align*}
and
$$\mathfrak{M}=  \frac{14}{3}.$$
Accordingly, if $\beta$ is such that
$$\frac{14}{3}<-\beta,$$
the couple $(n_1,n_2)$ belongs to $\mathbb{B}^\star_2$.
Hence, there exist four solutions of the form
$$
\begin{cases}
u = \alpha_{1} e_{1} + \alpha_{2} e_{2},\\\noalign{\vskip1.2mm}
\displaystyle
v = \frac{\sqrt{7}-4}{3}\, \alpha_{1} e_{1} + \frac{4\sqrt{7}+11}{3}\, \alpha_{2} e_{2},
\end{cases}
$$
where $\alpha_{1},\alpha_{2}\in\R$ solve the system
\begin{equation}
\label{essis}
\begin{cases}
\displaystyle
\alpha_1^2 + 4\alpha_2^2 = \frac{\sqrt{7}-4}{3} - 2- \beta,\\\noalign{\vskip1mm}
\displaystyle
\alpha_1^2 \Big(\frac{\sqrt{7}-4}{3}\Big)^2 + 4\alpha_2^2 \Big(\frac{4\sqrt{7}+11}{3}\Big)^2 = -\frac{4+\sqrt{7}}{3} - 2 - \beta,
\end{cases}
\end{equation}
and four solutions of the form
$$
\begin{cases}
u = \alpha_{1} e_{1} + \alpha_{2} e_{2},\\\noalign{\vskip1.2mm}
\displaystyle
v = -\frac{4+\sqrt{7}}{3}\, \alpha_{1} e_{1} + \frac{11-4\sqrt{7}}{3}\, \alpha_{2} e_{2},
\end{cases}
$$
where $\alpha_{1},\alpha_{2}\in\R$ solve the system
\begin{equation}
\label{essis2}
\begin{cases}
\displaystyle
\alpha_1^2 + 4\alpha_2^2 = -\frac{4+\sqrt{7}}{3} - 2- \beta,\\\noalign{\vskip1mm}
\displaystyle
\alpha_1^2 \Big(\frac{4+\sqrt{7}}{3}\Big)^2 + 4\alpha_2^2 \Big(\frac{11-4\sqrt{7}}{3}\Big)^2 =\frac{\sqrt{7}-4}{3} - 2 - \beta.
\end{cases}
\end{equation}
For instance, when $\beta=-5$, the solutions
of system~\eqref{essis} are
$$
(\pm \alpha_1,\pm\alpha_2)\and(\pm\alpha_1,\mp\alpha_2),
$$
with
\begin{align*}
&\alpha_1= -\frac13\sqrt{\frac{31(28+11\sqrt{7})}{35+16\sqrt{7}}}\approx-1.59482,\\\noalign{\vskip2mm}
&\alpha_2= -\frac16\sqrt{\frac{883 + 316\sqrt{7}}{18011+6808\sqrt{7}}}\approx-0.03587,
\end{align*}
while the solutions of system \eqref{essis2} are
$$
(\pm \alpha_1,\pm\alpha_2)\and(\pm\alpha_1,\mp\alpha_2),
$$
with
\begin{align*}
&\alpha_1= -\frac13\sqrt{\frac{31(11\sqrt{7}-28)}{16\sqrt{7}-35}}\approx-0.71992,\\\noalign{\vskip2mm}
&\alpha_2= -\frac16\sqrt{\frac{316\sqrt{7}-883}{6808\sqrt{7}-18011}}\approx-0.25809.
\end{align*}
\end{example}

%%%%%%%%%%%%%%%%%%%%%%%%%%%%%%%%%%%%%%%%%%%%%%%%%%

%%%%%%%%%%%%%%%%%%%%%%%%%%%%%%%%%%%%%%%%%%%%%%%%%%
\section{General Trimodal Solutions}

\noindent
Finally, we consider general trimodal solutions to system \eqref{MAIN}.
As previously shown, trimodal {\sc ee}-solutions exist. Then, one might ask
if system \eqref{MAIN} admits also trimodal solutions of not equidistributed energy.
The answer to this question is negative.

\begin{theorem}
Every trimodal solution is necessarily an {\sc ee}-solution.
\end{theorem}

\begin{proof}
Let $(u,v)$ be a (general) trimodal solution. In particular, with reference to \eqref{espressione},
$\alpha_{n_i}\not=0$ and $\gamma_{n_i}\not=0$ for every $n_i$.
Assume by contradiction that $(u,v)$ is not an {\sc ee}-solution. Then, in the light of Lemma \ref{VITTO}, the vectors
$$
\begin{bmatrix}
\alpha_{n_1}\\
\gamma_{n_1}\\
\end{bmatrix},\,
\begin{bmatrix}
\alpha_{n_2}\\
\gamma_{n_2}\\
\end{bmatrix},\,
\begin{bmatrix}
\alpha_{n_3}\\
\gamma_{n_3}\\
\end{bmatrix}
$$
are pairwise linearly independent. Accordingly, each of them can be written
as a linear combination of the other two. In particular, there exist $a,b,c,d,e,f\neq0$ such that
\begin{equation}
\label{sis1}
\begin{cases}
\alpha_{n_3} = a\alpha_{n_1} + b\alpha_{n_2},\\
\gamma_{n_3} = a\gamma_{n_1} + b\gamma_{n_2},
\end{cases}
\end{equation}
\begin{equation}
\label{sis2}
\begin{cases}
\alpha_{n_1} = c\alpha_{n_2} + d\alpha_{n_3},\\
\gamma_{n_1} = c\gamma_{n_2} + d\gamma_{n_3},
\end{cases}
\end{equation}
and
\begin{equation}
\label{sis3}
\begin{cases}
\alpha_{n_2} = e\alpha_{n_1} + f\alpha_{n_3},\\
\gamma_{n_2} = e\gamma_{n_1} + f\gamma_{n_3}.
\end{cases}
\end{equation}
Moreover, due to Lemma \ref{FILIPPO},
\begin{equation}
\label{CF}
\begin{cases}
\alpha_{n_1}+\gamma_{n_1}\not=0,\\
\alpha_{n_2}+\gamma_{n_2}\not=0,\\
\alpha_{n_3}+\gamma_{n_3}\not=0.
\end{cases}
\end{equation}
Therefore, recalling \eqref{rel},
\begin{align}
\label{formula1}
\displaystyle\lambda_{n_1} = - \frac{C_u\alpha_{n_1}+C_v\gamma_{n_1}}{\alpha_{n_1}+\gamma_{n_1}},\\\noalign{\vskip1mm}
\label{formula2}
\displaystyle\lambda_{n_2} = - \frac{C_u\alpha_{n_2}+C_v\gamma_{n_2}}{\alpha_{n_2}+\gamma_{n_2}},\\\noalign{\vskip1mm}
\label{formula3}
\displaystyle\lambda_{n_3} = - \frac{C_u\alpha_{n_3}+C_v\gamma_{n_3}}{\alpha_{n_3}+\gamma_{n_3}}.
\end{align}
Substituting the expressions of $\alpha_{n_3}$ and $\gamma_{n_3}$ given by \eqref{sis1}
into \eqref{formula3}, we obtain the identity
$$
[a(\alpha_{n_1}+\gamma_{n_1}) + b(\alpha_{n_2}+\gamma_{n_2})]\lambda_{n_3}=
- C_u[a\alpha_{n_1}+b\alpha_{n_2}]- C_v[a\gamma_{n_1}+b\gamma_{n_2}]
$$
which, making use of \eqref{formula1}-\eqref{formula2}, yields
\begin{equation}
\label{S1}
\textsf{A}\lambda_{n_1} + \textsf{B}\lambda_{n_2} = (\textsf{A}+\textsf{B})\lambda_{n_3}
\end{equation}
where
$$
\textsf{A} = a(\alpha_{n_1}+\gamma_{n_1}) \and \textsf{B} = b(\alpha_{n_2}+\gamma_{n_2}).
$$
An analogous reasoning, exploiting now \eqref{sis2} and \eqref{sis3}, provides the further equalities
\begin{align}
\label{S2}
\textsf{C}\lambda_{n_2} + \textsf{D}\lambda_{n_3} &= (\textsf{C}+\textsf{D})\lambda_{n_1},\\
\label{S3}
\textsf{E}\lambda_{n_1} + \textsf{F}\lambda_{n_3} &= (\textsf{E}+\textsf{F})\lambda_{n_2},
\end{align}
having set
\begin{align*}
&\textsf{C} = c(\alpha_{n_2}+\gamma_{n_2}),\\
&\textsf{D} = d(\alpha_{n_3}+\gamma_{n_3}),\\
&\textsf{E} = e(\alpha_{n_1}+\gamma_{n_1}),\\
&\textsf{F} = f(\alpha_{n_3}+\gamma_{n_3}).
\end{align*}
Since $a,b,c,d,e,f\not=0$, from \eqref{CF} we learn that
$\textsf{A},\textsf{B},\textsf{C},\textsf{D},\textsf{E},\textsf{F}\not=0$.
Then, introducing the matrix
$$
{\bf M}=
\begin{bmatrix}
\textsf{A} & \textsf{B} & -(\textsf{A}+\textsf{B})\\
-(\textsf{C}+\textsf{D}) & \textsf{C} & \textsf{D}\\
\textsf{E} & -(\textsf{E}+\textsf{F}) & \textsf{F}\\
\end{bmatrix}
$$
and the vector
$$
\boldsymbol{\lambda}=
\begin{bmatrix}
\lambda_{n_1} \\
\lambda_{n_2} \\
\lambda_{n_3} \\
\end{bmatrix},
$$
we rewrite \eqref{S1}-\eqref{S3} as
$$
{\bf M} \boldsymbol{\lambda}=\boldsymbol{0}.
$$
Direct calculations show that Det$({\bf M})=0$, thus Rank$({\bf M})<3$.

\smallskip
\noindent
$\diamond$ If Rank$({\bf M})=2$, in the light of the Rank-Nullity Theorem the solution set is a
one-dimensional linear subspace of $\R^3$, explicitly given by
$$
{\rm Ker}({\bf M})=
\left\{
\boldsymbol{\lambda}=
\begin{bmatrix}
\lambda \\
\lambda \\
\lambda \\
\end{bmatrix}: \lambda\in\R\right\}.
$$
In particular, this forces $\lambda_{n_1}=\lambda_{n_2}=\lambda_{n_3}$, implying the desired contradiction.

\smallskip
\noindent
$\diamond$ If Rank$({\bf M})=1$, there exists $\omega\not=0$ such that
$$
\begin{cases}
\textsf{A}=\omega \,\textsf{B},\\
(1+\omega)\textsf{C}=\textsf{D}.
\end{cases}
$$
Substituting the explicit expressions of $\textsf{A},\textsf{B},\textsf{C},\textsf{D}$ into the
system above
\begin{align}
\label{U1}
&a(\alpha_{n_1}+\gamma_{n_1})= \omega b(\alpha_{n_2}+\gamma_{n_2}),\\
\label{U2}
&c(1+\omega)(\alpha_{n_2}+\gamma_{n_2})= d(\alpha_{n_3}+\gamma_{n_3}).
\end{align}
Then, plugging \eqref{sis1} into \eqref{U2} and exploiting \eqref{U1} and \eqref{CF},
$$
c(1+\omega) = db (1+\omega).
$$
Since $1+\omega\neq0$ (due to the fact that $\textsf{D}\neq0$), we end up with
$$
c=db.
$$
Appealing now to \eqref{sis1} and \eqref{sis2},
$$
(1+da)
\begin{bmatrix}
\alpha_{n_1}\\
\gamma_{n_1}
\end{bmatrix}
= 2d \begin{bmatrix}
\alpha_{n_3}\\
\gamma_{n_3}
\end{bmatrix},
$$
meaning that the two vectors
$$
\begin{bmatrix}
\alpha_{n_1}\\
\gamma_{n_1}
\end{bmatrix}\and
 \begin{bmatrix}
\alpha_{n_3}\\
\gamma_{n_3}
\end{bmatrix}
$$
are linearly dependent.
\end{proof}

\begin{example}
As a particular case, let us consider
$$
A = L^{\frac{p+1}{2}}, \quad p\in\N,
$$
with $L$ as in Example~\ref{dlo}.
In this situation, the eigenvalues read
$$
\lambda_n = n^{p+1}\pi^{p+1}.
$$
Accordingly, given a trimodal solution (which, as we know, is
necessarily an {\sc ee}-solution) and exploiting Corollary \ref{WWW}, we deduce the relation
$$
n_1^{p+1}+n_2^{p+1}=n_3^{p+1}.
$$
Therefore, when $p=1$, they
form a Pythagorean triplet. Otherwise the identity is impossible, due to the
celebrated {\it Fermat's Last Theorem} proved by A.\ Wiles in recent years \cite{TW,WIL}.
Hence, for $p=2,3,4,\ldots,$ trimodal solutions do not exist.
\end{example}
%%%%%%%%%%%%%%%%%%%%%%%%%%%%%%%%%%%%%%%%%%%%%%%%%%

%%%%%%%%%%%%%%%%%%%%%%%%%%%%%%%%%%%%%%%%%%%%%%%%%
\section{Comparison with Single-Beam Equations}

\noindent
We conclude by comparing our results on the double-beam system \eqref{MAIN}
with some previous achievements on extensible single-beam equations.
As customary, along the section, we will set
\begin{equation}
\label{bis}
C_u=\beta + \cz \|u\|_1^2.
\end{equation}

The following theorem has been proved in \cite{CZGP}.

\begin{theorem}
\label{SS1}
The nontrivial solutions of the single-beam equation
$$
A u + C_u u = 0
$$
are exactly $2|\mathbb{E}|$, where, in the usual notation,
$$
\mathbb{E} = \{n: \lambda_n < - \beta \}
$$
denotes the (finite) set of effective modes. Such solutions are unimodal, explicitly given by
$$
u_n^{\pm} = \pm \sqrt{\frac{- \beta - \lambda_n}{\cz\lambda_n}}\, e_n,
$$
for every $n\in\mathbb{E}$.
\end{theorem}

Concerning the case of single beams which rely on an elastic foundation, the result reads as follows.

\begin{theorem}
\label{bocch}
The nontrivial solutions of the single-beam equation
\begin{equation}
\label{bvv}
A^2 u + C_u Au + ku = 0
\end{equation}
can be either unimodal or bimodal (but not trimodal). In addition, the following hold.

\begin{itemize}
\item Equation \eqref{bvv} admits nontrivial unimodal solutions
if and only if the set
$$
\mathbb{F} = \left\{n: \frac{k}{\lambda_n} + \lambda_n < -\beta \right\}
$$
is nonempty. More precisely, for every $n\in \N$, one of the following disjoint situations occurs.

\begin{itemize}
\item If $n\in\mathbb{F}$, we have exactly 2 nontrivial unimodal solutions of the form
$$
u_n^{\pm} = \pm \sqrt{\frac{1}{\cz\lambda_n}\left(-\beta - \frac{k}{\lambda_n}-\lambda_n\right)}\, e_n.
$$
\item If $n\notin\mathbb{F}$ all the unimodal solutions involving the eigenvector $e_{n}$ are trivial.
\end{itemize}

\smallskip
\item Equation \eqref{bvv} admits nontrivial bimodal solutions if and only if the set
$$
\mathbb{G} = \{(n_1,n_2): n_1<n_2,\, \lambda_{n_1}+\lambda_{n_2} < -\beta \,\text{ and }\, \lambda_{n_1}\lambda_{n_2}=k \}
$$
is nonempty. More precisely, for every couple $(n_1,n_2) \in \N$ with $n_1<n_2$, one of the following
disjoint situations occurs.
\begin{itemize}
\item If $(n_1,n_2) \in \mathbb{G}$, we have exactly the (infinitely many) solutions of the form
$$
u= x e_{n_1} + y e_{n_2},
$$
for all $(x,y)\in\R^2$ satisfying the equality
$$
\cz x^2\lambda_{n_1} + \cz y^2\lambda_{n_2} + \lambda_{n_1} +\lambda_{n_2} +\beta = 0 \qquad \text{with}\qquad xy\neq0.
$$
\item If $(n_1,n_2) \notin \mathbb{G}$, there are no nontrivial bimodal solutions involving the eigenvectors
$e_{n_1}$ and $e_{n_2}$.
\end{itemize}
\end{itemize}
\end{theorem}

Theorem \ref{bocch} has been proved in \cite{BoV}, in the concrete situation when $A=L$ (the Laplace-Dirichlet operator).
We present here a short proof, which is valid even in our abstract setting.

\begin{proof}[Proof of Theorem \ref{bocch}]
Let $u$ be a weak solution\footnote{Analogously to \eqref{weak}, $u\in\H^2$ is called
a {\it weak solution} to \eqref{bvv} if, for every test $\phi\in \H^2$,
$$
\l u, \phi\r_2 + C_u \l u, \phi \r_1 + k \l u, \phi \r = 0.
$$
}
to \eqref{bvv}. Arguing as in the proof of Lemma \ref{gene}, that is, writing
$$
u = \sum_{n} \alpha_n e_n
$$
for some $\alpha_n\in\R$, we obtain, for every $n\in\N$, the identity
$$
\lambda_n^2 \alpha_n + C_u\lambda_n\alpha_n + k\alpha_n = 0.
$$
Hence, if $\alpha_n\not=0$, we infer that
$$
\lambda_n^2 + C_u\lambda_n  + k = 0.
$$
Since the equation above admits at most two distinct solutions
$\lambda_{n_i}$, we conclude that the nontrivial solutions to equation \eqref{bvv} can be either
unimodal or bimodal (but not trimodal).

First, let us look for unimodal solutions $u$ of the form
$$
u = \alpha_n e_n
$$
for a fixed $n\in\N$ and some coefficient $\alpha_n\neq0$. Analogously to the proof
of Theorem~\ref{unimodal}, from \eqref{bvv} we obtain
$$
\lambda_n^2 + (\beta + \cz\lambda_n \alpha_n^2)\lambda_n + k = 0,
$$
which implies
$$
\alpha_n^2 = \frac{1}{\cz\lambda_n} \Big(-\beta - \frac{k}{\lambda_n}-\lambda_n\Big).
$$
Therefore, there exist nontrivial unimodal solutions (explicitly computed) if and only if $n\in\mathbb{F}$.

Next, let us look for bimodal solutions $u$ of the form
$$
u=\alpha_{n_1} e_{n_1} + \alpha_{n_2} e_{n_2}
$$
with $n_1<n_2\in\N$ and $\alpha_{n_i}\in\R\setminus\{0\}$.
Similarly to the previous situation, from \eqref{bvv} we obtain the system
$$
\begin{cases}
\lambda_{n_1}^2 + C_u \lambda_{n_1} + k = 0,\\
\lambda_{n_2}^2 + C_u \lambda_{n_2} + k = 0.
\end{cases}
$$
Hence
$$
\lambda_{n_1}\lambda_{n_2}=k
$$
and the value $C_u$ is determined by \eqref{bis}, which yields the relation
$$
\cz\alpha_{n_1}^2 \lambda_{n_1} + \cz\alpha_{n_2}^2 \lambda_{n_2} +\lambda_{n_1} + \lambda_{n_2} +\beta = 0.
$$
Therefore, there exist nontrivial bimodal solutions (explicitly computed) if and only if $(n_1,n_2)\in\mathbb{G}$.
\end{proof}

A closer look to Theorems \ref{SS1} and \ref{bocch} reveals
that the set of steady states of the double-beam system \eqref{MAIN}
is very rich, and by no means represents a
``double-copy" of the set of stationary solutions of a single-beam equation:

\begin{itemize}
\item According to \S\ref{sezioneUnimodal}, nonsymmetric unimodal solutions pop up, as well
as unimodal solutions for
which the elastic energy is not evenly distributed. This feature is
illustrated in the forthcoming pictures\footnote{The notation in the captions is the same as in
\S\ref{sezioneUnimodal}.}. Moreover, not only a double series of
bifurcations of the trivial solution occurs, but even buckled unimodal solutions suffer from a
further bifurcation (see Lemma~\ref{reale} and Fig.~\ref{secondafigura} of \S\ref{sezioneUnimodal}).

\smallskip
\item According to \S\ref{EES} and \S\ref{GB}, system \eqref{MAIN} admits infinitely many bimodal and
trimodal {\sc ee}-solutions, and also
finitely many nonsymmetric bimodal solutions of not equidistributed energy.
\end{itemize}

\begin{figure}[ht]
\includegraphics[width=13cm]{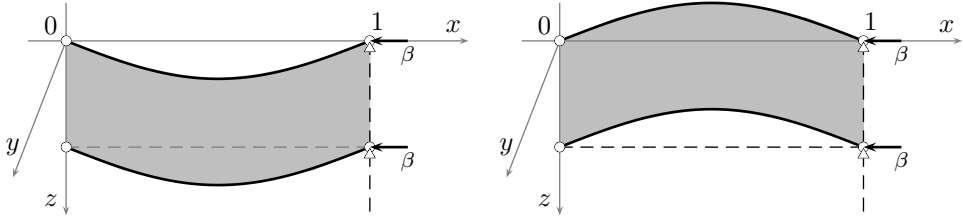}
\caption{Symmetric in-phase unimodal solutions $(\alpha_{1,1}^\pm,\alpha_{1,1}^\pm)$.}
\end{figure}

\begin{figure}[ht]
\includegraphics[width=13cm]{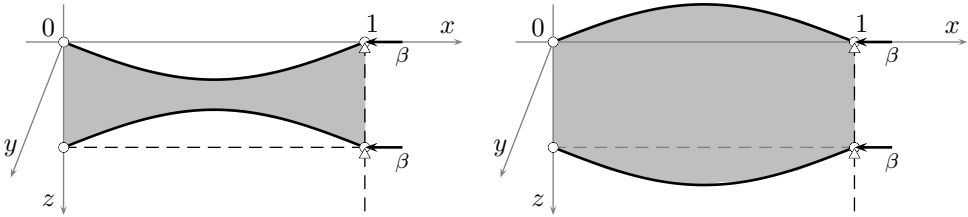}
\caption{Symmetric out-of-phase unimodal solutions $(\alpha_{1,2}^\pm,\alpha_{1,2}^\mp)$.}
\end{figure}

\begin{figure}[ht]
\includegraphics[width=13cm]{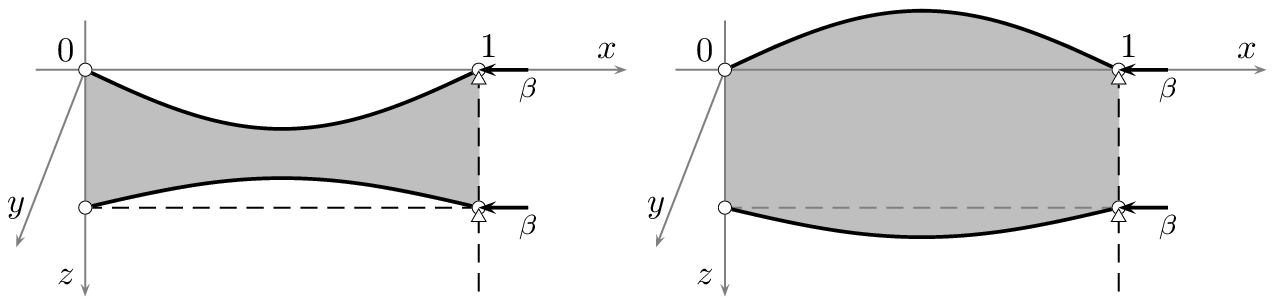}
\caption{Nonsymmetric out-of-phase unimodal solutions $(\alpha_{1,3}^\pm,\alpha_{1,4}^\mp)$.}
\end{figure}

\begin{figure}[ht]
\includegraphics[width=13cm]{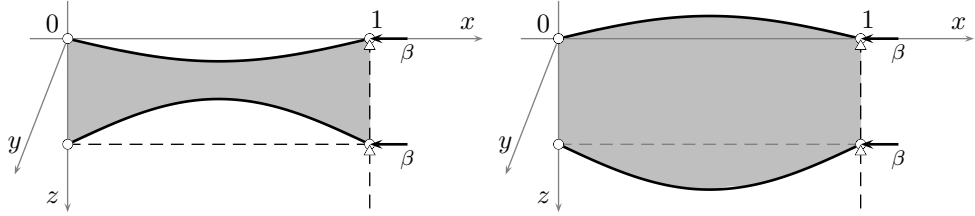}
\caption{Nonsymmetric out-of-phase unimodal solutions $(\alpha_{1,4}^\pm,\alpha_{1,3}^\mp)$.}
\end{figure}

%%%%%%%%%%%%%%%%%%%%%%%%%%%%%%%%%%%%%%%%%%%%%%%%%

%%%%%%%%%%%%%%%%%%%%%%%%%%%%%%%%%%%%%%%%%%%%%%%%%
\section*{Appendix:\ Dimensionless Models of Double-Beam Systems}
\setcounter{equation}{0}
\setcounter{subsection}{0}
\newtheorem*{remAPP}{Remark}

\noindent
Let us consider a thin and elastic Woinowsky-Krieger beam of natural length $\ell>0$,
uniform cross section $\Omega$,
and thickness $0<h\ll\ell$. The beam is supposed to be homogeneous, of constant mass density $\rho>0$
per unit volume, and
symmetric with respect to the vertical plane ($\xi$-$z$). Hence,
we can restrict our attention
to its rectangular section lying in the plane $y=0$.
Identifying the beam with such a section, we assume that
its middle line at rest occupies the interval $[0,\ell]$ of the $\xi$-axis.
According to the physical analysis carried out in \cite{CZGP,GN},
in the isothermal case the motion equation for the vertical deflection
of the midline of the beam
$$U:(\xi,\tau)\in[0,\ell]\times \R^+ \mapsto \R$$
reads
$$
\mathfrak{L} U
-\frac{E h}{2\ell^2(1-\nu^2)}\bigg(2D+
\int_0^\ell |\partial_\xi U(s)|^2\, \d s \bigg)\partial_{\xi\xi} U = \frac{G}{\ell|\Omega|}.
$$
Here,
$$
\mathfrak{L}=\rho \partial_{\tau\tau} - \frac{\rho h^2}{12}\partial_{\tau\tau\xi\xi}+
\frac{E h^3}{12\ell(1-\nu^2)}\,\partial_{\xi\xi\xi\xi}
$$
denotes the evolution operator, while
\smallskip
\begin{itemize}
\item $|\Omega|>0$ is the area of the cross section,
\smallskip
\item $E>0$ is the Young modulus (force per unit area),
\smallskip
\item $\nu\in(-1,\frac12)$ is the Poisson ratio, which is negative for auxetic materials,
\smallskip
\item $D\in\R$ is the axial displacement at the right end of the beam,
\smallskip
\item $G:[0,\ell]\times \R^+ \to \R$ is the vertical body force applied on the section $\Omega$.
\end{itemize}

\smallskip
\noindent
We point out that the model is obtained by supposing the beam slender (i.e.\ $h\ll \ell$),
and the modulus of the axial
displacement $D$ small when compared to the length of the beam (i.e.\ $|D|\ll \ell$ as well).
See also \cite{CIA,CIAGRA,LL} for more details.

\smallskip
Assuming that $G$ is due to the distributed and
mutual elastic action exerted between two equal Woinowsky-Krieger beams
with vertical deflections $U=U(\xi,\tau)$ and $V=V(\xi,\tau)$,
respectively, we let
$$
G(\xi,\tau) = - \varkappa\big[U(\xi,\tau)-V(\xi,\tau)\big],
$$
being $\varkappa>0$ the uniform stiffness (force per unit length) of the elastic core.
In this situation, the model describing the motion of the resulting elastically-coupled
extensible double-beam nonlinear system becomes
$$
\begin{cases}
\displaystyle
\mathfrak{L}U
-\frac{E h}{2\ell^2(1-\nu^2)}\Big(2D+
\int_0^\ell |\partial_\xi U(s)|^2\, \d s\Big)\,\partial_{\xi\xi} U +
\frac{\varkappa}{\ell|\Omega|}(U-V)=0,\\\noalign{\vskip1.3mm}
\displaystyle
\mathfrak{L}V
-\frac{E h}{2\ell^2(1-\nu^2)}
\Big(2D+\int_0^\ell |\partial_\xi V(s)|^2\, \d s\Big)\,\partial_{\xi\xi} V
- \frac{\varkappa}{\ell|\Omega|}(U-V)=0.
\end{cases}
$$
In order to rewrite the system in dimensionless form,
we exploit the fact that the two beams have the same structural parameters.
In particular, $\ell$ is viewed as the common {\it characteristic length} of the beams, while
the {\it characteristic time} $\tau_0$ is obtained by means
of the well-known shear wave velocity $c_0$ in bulk
elasticity, given by
$$
c_0=\sqrt{\frac{E}{2\rho(1+\nu)}}.
$$
Then, the characteristic time $\tau_0$ is equal to the ratio $\ell/c_0$. Explicitly,
$$\tau_0=\sqrt{\frac{2 \ell^2\rho(1+\nu)}{E}}.$$
Consequently, introducing the dimensionless space and time variables
$$
x= \frac{\xi}{\ell}\in[0,1]\and
t=\frac{\tau}{\tau_0}\in\R^+,
$$
along with the rescaled unknowns $u,v:[0,1]\times\R^+\to \R$ defined as
$$
u(x,t)= \frac{U(\ell x,\tau_0 t)}{\ell}\and
v(x,t)= \frac{V(\ell x,\tau_0 t)}{\ell},
$$
we end up with the dimensionless model
$$
\begin{cases}
\displaystyle
\frac{\ell (1-\nu)}{h}\Big(\partial_{tt}
- \frac{h^2}{12\ell^2}\partial_{ttxx}\Big)u +
\delta\partial_{xxxx} u
-\big(\chi +\|\partial_x u \|^2\big)\partial_{xx} u +\kappa (u-v)=0,\\\noalign{\vskip1.7mm}
\displaystyle
\frac{\ell (1-\nu)}{h}\Big(\partial_{tt}
- \frac{h^2}{12\ell^2}\partial_{ttxx}\Big) v + \delta\partial_{xxxx} v
-\big(\chi +\|\partial_x v \|^2\big)\partial_{xx} v -\kappa(u-v)=0,
\end{cases}
$$
where $\|\cdot\|$ denotes the $L^2$-norm on the unit interval $[0,1]$,
and
$$
\delta=\frac{h^2}{6\ell^2}>0,\qquad
\chi= \frac{2D}{\ell}\in\R,\qquad
\kappa=\frac{2\varkappa\ell^2(1-\nu^2)}{E |\Omega|h}>0.
$$
Under reasonably physical assumptions on the stiffness $\varkappa$ of the elastic core, and since $D$
and $h$ are comparable, we may conclude that $|\chi|$ and $\kappa$ share the same order of magnitude $h/\ell$,
whereas $\delta$ is much smaller. Accordingly,
$|{\chi}/{\delta}|$ and ${\kappa}/{\delta}$ may assume large values, 
for their order of magnitude is $\ell/h\gg1$. Hence, all the stationary solutions 
exhibited in this paper are physically consistent.
%%%%%%%%%%%%%%%%%%%%%%%%%%%%%%%%%%%%%%%%%%%%%%%%%

%%%%%%%%%%%%%%%%%%%%%%%%%%%%%%%%%%%%%%%%%%%%%%%%%

%%%%%%%%%%%%%%%%%%%%%%%%%%%%%%%%%%%%%%%%%%%%%%%%%%
\end{document}